\documentclass[11pt, reqno]{amsart}

\usepackage{amssymb,amsmath,amscd,graphicx,
latexsym,amsthm}
\usepackage{amssymb,latexsym,amsmath,amscd,graphicx}
\usepackage[mathscr]{eucal}

\usepackage{scalerel,stackengine}
\DeclareMathSymbol{\shortminus}{\mathbin}{AMSa}{"39}

\stackMath
\newcommand\reallywidehat[1]{%
\savestack{\tmpbox}{\stretchto{%
  \scaleto{%
    \scalerel*[\widthof{\ensuremath{#1}}]{\kern.1pt\mathchar"0362\kern.1pt}%
    {\rule{0ex}{\textheight}}
  }{\textheight}%
}{2.4ex}}%
\stackon[-6.9pt]{#1}{\tmpbox}%
}
 \input xy
 \xyoption{all}
\def\ZZ{{\mathbb Z}}
 
        \def\Q{{\mathcal Q}}

\def\OO{{\mathcal O}}
\def\R{{\mathbf R}}

\def\F{\mathcal{F}}
\def\E{\mathcal{E}}
\def\G{\mathcal{G}}

\def\U{\mathcal{U}}

\def\cU{\mathcal{U}}
\def\cP{\mathcal{P}}
\def\Pic0{{\rm Pic}^0}
\def\Proj{{\rm Proj} \, }
\def\Aut0{{\rm Aut}^0}

\def\R{{\mathbf{R}}}

\def\*{{\underline *}}

\def\Supp{\mathrm{Supp}}

\def\Alb{{\rm Alb}\,}
\def\id{{\rm id}}
\def\lotimes{\stackrel{\mathbf L}{\otimes}}
\def\FM{{\rm FM}}

\DeclareMathOperator{\ad}{ad}

\theoremstyle{plain}

\newtheorem{theorem}{Theorem}[subsection]
\newtheorem{theoremalpha}{Theorem}

\newtheorem{proposition/example}[theorem]{Proposition/Example}
\newtheorem{proposition}[theorem]{Proposition}

\newtheorem{corollary}[theorem]{Corollary}
\newtheorem{lemma}[theorem]{Lemma}
\newtheorem{variant}[theorem]{Variant}
\newtheorem{claim}[theorem]{Claim}

\theoremstyle{definition}

\newtheorem{remark}[theorem]{Remark}

\newtheorem{conjecture/question}[theorem]{Conjecture/Question}

\newtheorem{remark/definition}[theorem]{Remark/Definition}
\newtheorem{notation/assumptions}[theorem]{Assumptions/Notation}
\newtheorem{setting/notation}[theorem]{Setting and Notation}

\numberwithin{equation}{section}

 \theoremstyle{remark}

\begin{document}
\title{Derived invariance of the Albanese relative canonical ring}
 \author[ F. Caucci, L. Lombardi, G. Pareschi]{ Federico Caucci, Luigi Lombardi and  Giuseppe Pareschi}

\address{F. Caucci \\Dipartimento di Matematica,  
             Universit\`a degli Studi di Milano Statale\\Via Cesare Saldini 50, 20133 Milano, Italy}
 \email{federico.caucci@unimi.it }
\thanks{FC was  supported by the ERC Consolidator Grant ERC-2017-CoG-771507-StabCondEn.}

\address{L. Lombardi\\Dipartimento di Matematica, 
             Universit\`a degli Studi di Milano Statale\\Via Cesare Saldini 50, 20133, Milano, Italy}
\email{luigi.lombardi@unimi.it}
\thanks{LL was supported by GNSAGA, PSR Linea 4, PRIN 2020: Curves, Ricci flat Varieties and their Interactions}

\address{G. Pareschi\\Dipartimento di Matematica,
              Universit\`a di Roma Tor Vergata\\Viale della Ricerca Scientifica, 00133 Roma, Italy}
\email{pareschi@mat.uniroma2.it}
 \thanks{GP was partially supported by  the MIUR Excellence Department Project MatMod@TOV awarded to the Department of Mathematics, University of Rome Tor Vergata.}

\begin{abstract} 
We show the derived invariance of various geometric invariants of smooth complex projective varieties governed by the Albanese map, including the relative canonical ring and the class of the relative canonical model in a suitable variant of the Grothendieck ring of varieties.  Then we derive some applications to the derived invariance of Hodge numbers. 
\end{abstract}
\maketitle

\section{Introduction}\label{I0}

This paper is concerned with the understanding of which aspects 
 of the geometry of a smooth complex projective variety  are preserved under exact equivalences of its derived category.  One of the most fundamental results in this direction is the derived invariance of the canonical ring
\[R(X)  \; = \; \bigoplus_{m \geq 0} H^0 ( X , \omega_X^{\otimes m})\]
(Orlov \cite[Corollary 2.1.9]{orlov}). In the first part of this paper we prove a similar result for the relative canonical ring under the Albanese morphism $a_X \colon X\rightarrow \Alb X$
\[
\mathcal R(a_X)=\bigoplus_{m \geq 0}{a_X}_*\omega_X^{\otimes m}.
\]
The second part of the paper
is devoted to  applications to the derived invariance of various geometric data. In particular we focus on motivic classes of irregular varieties in suitable variants of the Grothendieck ring of varieties introduced in the paper \cite{itoetal}, and on Hodge numbers.

Turning to a more detailed presentation, we recall that an exact equivalence $\Phi \colon D(X)\buildrel\sim\over\rightarrow D(Y)$  induces canonically  an isomorphism  of algebraic groups  $\varphi \colon
 \Aut0 X\times \Pic0 X\buildrel\sim\over\rightarrow \Aut0Y\times \Pic0 Y$, referred to as the Rouquier isomorphism  \cite{rouquier}.  
By employing this  isomorphism, in \cite[Theorem A]{ps1} 
the authors show that the Picard varieties
 $\Pic0 X$ and $\Pic0 Y$ are (non-canonically)
 isogenous abelian varieties, but  are  in general not
  isomorphic. Hence the same happens for the Albanese varieties. 
 In conclusion the Albanese relative canonical rings $\mathcal R(a_X)$ and $\mathcal R(a_Y)$ are sheaves over varieties which are in general non-isomorphic. 
 
 Our main result is that nevertheless $\mathcal R(a_X)$ and $\mathcal R(a_Y)$ are somewhat preserved by the (dual) Rouquier isomorphism. 
This is accomplished in a series of steps, of independent interest, leading to an auxiliary version of the Albanese relative canonical ring,  better behaving  under exact equivalence, as described below.

 We first consider the Iitaka fibration of $X$ and more precisely one of its  models 
\[\widetilde X\rightarrow Z_X\>
\]
 such that $Z_X$ is smooth. The abelian subvariety $\Pic0 Z_X\buildrel{i}\over\hookrightarrow \Pic0 X$ does not depend on the model. Moreover the image of the Albanese map $a_{Z_X}\colon Z_X\rightarrow \Alb Z_X$ coincides with the image of the natural morphism $c_X$ (referred to as the \emph{Albanese-Iitaka map})
\begin{equation}\label{diagram}
\xymatrix{
 X\ar[r]^-{a_X}\ar[rd]_{c_X}&\Alb X\ar[d]^{p_X}\\
 &\Alb Z_X}
\end{equation}
where $p_X$ is the projection dual of  the inclusion $i$ (see Subsection \ref{R3}). Our first result is that the Albanese-Iitaka variety  and the image of the Albanese-Iitaka map are preserved by the (dual) Rouquier isomorphism. In fact a more precise result holds true. This can be summarized as follows (see Theorem \ref{step1} for the precise statement).

\begin{theoremalpha}\label{A} 
 Let $\Phi \colon D(X)\rightarrow D(Y)$ be an exact equivalence. Then: 
\begin{itemize}
\item[(a)] The subvariety $\Pic0 Z_X$ of $\Pic0 X$ is a derived invariant. More precisely, the Rouquier isomorphism sends $\Pic0 Z_X$ to $\Pic0 Z_Y$ inducing an isomorphism $\varphi \colon \Pic0 Z_X\buildrel\sim\over\longrightarrow \Pic0 Z_Y$. \\

\item[(b)] The dual isomorphism 
\begin{equation}\label{eq:isodualR}
\widehat\varphi \colon \Alb Z_Y\buildrel\sim\over\longrightarrow \Alb Z_X
\end{equation}
 sends  $c_Y(Y)$ to $c_X(X)$ \emph{(up to translation)}. \\

\item[(c)] More precisely, let
\[
\xymatrix{
 X\ar[r]^{s_X} \ar[r]\ar[rd]_-{c_X}&X^\prime\ar[d]^{c_X^\prime}\\
 &\Alb Z_X}
\] 
be the Stein factorization of the Albanese-Iitaka morphism $c_X$. 
By employing  the same notation for the Stein factorization of the Albanese-Iitaka morphism  of $Y$, there is an isomorphism 
$ \psi:Y^\prime\buildrel\sim\over\longrightarrow X^\prime$ such that the diagram
\[
\xymatrix{X^\prime\ar[d]^{c_X^\prime}&Y^\prime\ar[d]^{c_Y^\prime}\ar[l]^{\psi}_\sim\\
\Alb Z_X&\Alb Z_Y\ar[l]_-{\widehat\varphi}}
\]
is commutative.
\end{itemize}
\end{theoremalpha}

Note: in the upcoming paper \cite[Theorem 3.0.1 $(iii)$ and Remark 3.0.5]{calo}, 
the authors  prove further that the general fibers of   
$s_X$ and $s_Y$ are derived equivalent. Moreover, the equivalence 
$\Phi$ becomes  a relative   equivalence
when restricted to the loci where both the morphisms $\psi^{-1} \circ s_X$ and $s_Y$ are flat
(however these further properties are not necessary in  this paper).

The previous theorem enables us to prove the derived invariance of the relative canonical ring 
attached to the Albanese-Iitaka morphism:
$$
\mathcal R(c_X) \; = \; \bigoplus_{m\ge 0}{c_X}_*\omega_X^{\otimes m}.
$$
More generally, we prove the invariance in the derived category of the complex
\begin{equation}\label{eq:relstr}
\mathcal {HA}(c_X) \; = \;  \bigoplus_{m \in \mathbb Z , \, \, q \geq 0} \mathbf R\,  {c_X}_*\bigl(   \Omega_X^q \otimes \omega_X^{\otimes m} \bigr) [q],\>
\end{equation}
whose cohomology can be thought as a relative version of the 
bigraded algebra 
$$\bigoplus_{m \in \mathbb Z, \, k }\;\; \bigoplus_{p-q=k} H^p ( X , \Omega_X^q \otimes \omega_X^{\otimes m})$$ 
in \cite[Corollary 2.1.10]{orlov}.
We refer to  \S\ref{HHoch} and \S\ref{Hcan} for the proof of the following theorem.

\begin{theoremalpha}\label{B} 
Let $\Phi \colon D(X) \to D(Y)$ be an exact equivalence.
Then the  isomorphism \eqref{eq:isodualR} induces:
\begin{itemize}
\item[(a)] an isomorphism of $\OO_{\Alb Z_Y}$-algebras 
\[\widehat{\varphi}^*\mathcal R(c_X) \; \cong \; \mathcal R(c_Y);\] 

\item[(b)]
isomorphisms of $\OO_{\Alb Z_Y}$-modules
\[
\bigoplus_{p-q=k} \widehat{\varphi}^* R^{p}{c_X}_*(\Omega_X^q\otimes \omega_X^{\otimes m})  \; \cong\;
\bigoplus_{p-q=k}R^{p}{c_Y}_*(\Omega_Y^q\otimes \omega_Y^{\otimes m}) \quad \quad  \forall \,  m, k \in \mathbb Z.
\]

\end{itemize}
\end{theoremalpha}

Now  it is time to state the announced result regarding  the derived  invariance of the relative canonical algebra with respect to the Albanese map.  Roughly speaking, we prove that the Albanese relative canonical  algebra is the pullback of a derived invariant algebra on the Albanese-Iitaka variety ``weighted" by a derived invariant (finite) semigroup of topologically trivial line bundles. 

Less informally, we prove two statements. The first one (which is independent on derived equivalences), is  Theorem \ref{C}(a) below. Roughly, its content is that the relative canonical ring $\mathcal R(a_X)$ is very close to being the pullback of a (finitely generated)  graded $\OO_{\Alb Z_X}$-algebra, denoted by $\mathcal U_X$,   via the quotient map $p_X$ of \eqref{diagram}. More precisely, the graded components of the algebra $\mathcal U_X$ are endowed with a canonical decomposition, and the relative canonical algebra $\mathcal R(a_X)$ is obtained from the pullback of $\mathcal U_X$  by twisting the pullback of the summands  appearing in such decompositions by  torsion line bundles belonging to a certain finite subgroup $S_X$ of $\Pic0 X$. 

The second statement is Theorem \ref{C}(b), telling  that both the twisting semigroup $S_X$ and the algebra $\U_X$ are derived invariants. More precisely $S_X$ and $\U_X$ are sent to  $S_Y$ and $\U_Y$, respectively,  by the Rouquier  isomorphism  and by the dual isomorphism (\ref{eq:isodualR}).  
This second statement is in fact a variant of Theorem \ref{B}(a). The precise statement (see Theorem \ref{thmbasechange}) is as follows.

\begin{theoremalpha}\label{C} 

\begin{itemize}
\item[(a)] For any smooth projective variety $X$, and for all integers $m>0$, there exists an essentially canonical decomposition 
\[
{a_X}_*\omega_X^{\otimes m}\cong \bigoplus_{i=1}^{N_X(m)} \big(p_X^*\G_{X,m,i} \big) \otimes P_{\alpha_{X,m,i}}\>,
\] 
where $P_{\alpha_{X,m,i}}$ are line bundles parametrized by torsion points  $\alpha_{X,m,i}\in\Pic0 X$.
Moreover, the sheaves 
\[\mathcal U_X^m: =\bigoplus_{i=1}^{N_X(m)}\G_{X,m,i}
\]
 are the graded pieces of a finitely generated $\OO_{\Alb Z_X}$-algebra $\mathcal U_X$.\\

\item[(b)] Let $\Phi \colon D(X) \to D(Y)$ be an exact equivalence. For every $m>0$ we have $N_X(m)=N_Y(m)$. Moreover, by 
 keeping the setting and notation of Theorems \ref{A} and \ref{B}, we obtain
\[\varphi(\mathrm{id}_X,\alpha_{X,m,i})=(\mathrm{id}_Y,\alpha_{Y,m,i})
\]
(up to reordering). 
Finally, there is an isomorphism 
\[\widehat{\varphi}^*\mathcal U_X\cong \mathcal U_Y
\]
as $\OO_{\Alb Z_Y}$-graded algebras.
\end{itemize}
\end{theoremalpha}

Note that Orlov's theorem on the derived invariance of the canonical ring is recovered both from Theorem \ref{C} and Theorem \ref{B}(a) by taking $H^0$. 

Finally, we remark that in many circumstances the Rouquier isomorphism already sends $\Pic0 X$ to $\Pic0 Y$.  For example this happens when $\Aut0 X$ is affine, or $\chi(X,\OO_X)\ne 0$.  In this case the statement of Theorem \ref{C} is enormously simplified, because the Rouquier isomorphism  induces a dual isomorphism $\widehat{\varphi}:\Alb Y\rightarrow \Alb X$ so that $\widehat{\varphi}^*\mathcal R(a_X)\cong \mathcal R(a_Y)$, see Corollary \ref{cor:invariance of ra}.

In the second part of the paper we draw the following applications of Theorem \ref{C}. 

\noindent{\bf Kodaira dimension. } It follows from  Theorem \ref{C} that the Kodaira dimension of a general fiber of the Albanese map is a derived invariant. More precisely (see Corollary \ref{cor:koddim}), we obtain an  isomorphism 
between the  canonical rings of the general fibers over points of $\Alb X$ (resp. $\Alb Y$) mapping to points of $\Alb Z_X$ (resp. $\Alb Z_Y$) corresponding  via the isomorphism $\widehat{\varphi}$.

\noindent{\bf Motivic classes and Hodge numbers. } More refined derived invariants  are obtained  in the case where the canonical sheaf is relatively big (respectively big and nef) with respect to the Albanese map. In the spirit of a recent paper of Ito-Miura-Okawa-Ueda \cite{itoetal}, we show  that in this case suitable motivic classes attached to a pair of 
  derived equivalent varieties $X$ and $Y$ are equal in suitable variants of the Grothendieck ring of varieties \emph{modulo  isogeneous abelian varieties} (see \S \ref{E1} and \S \ref{lastsection0}). More precisely,
we consider the relative canonical model (\cite{bchm}) with respect to the Albanese map
\[
X^{can}_{a}=\Proj \mathcal R(a_X)\rightarrow \Alb X\>.
\]

\begin{theoremalpha}\label{Di} Let $X$ and $Y$ be smooth projective varieties over $\mathbb C$ such that $D(X)\cong D(Y)$. Then:

\begin{itemize}
\item[(a)] \ $[X^{can}_{a}] = [Y^{can}_{a}]$  in the Grothendieck ring $K^\prime_0(\mathrm{Var}/{\mathbb{C}})$ of varieties  over $\mathbb C$  modulo isogeneous abelian varieties  \emph{(see Theorem  \ref{thmGroclass})}. \\

\item[(b)]  If $\omega_X$ is $a_X$-big and nef, then $[X]=[Y]$ in the localized and completed  Grothendieck ring modulo isogeneous abelian varieties $\widehat{\mathcal{M}_{\mathbb{C}}^{\prime}}$
 \emph{(see Theorem \ref{invminimal})}. 
\end{itemize}
\end{theoremalpha} 

 In \cite[Problem 7.2]{itoetal} the authors ask whether derived equivalent varieties have the same class in  $\widehat{\mathcal{M}_{\mathbb{C}}^{\prime}}$. 
 Theorem \ref{Di}(b)  answers affirmatively  their question when $\omega_X$ is $a_X$-big and nef.
We also refer the reader to the  formulation under more general hypotheses provided by Theorem \ref{thm:minimal} below.

By taking virtual Hodge realizations and Hodge-Deligne polynomials, this has the following consequence about the invariance of Hodge numbers (Theorems  \ref{invHN1} and \ref{invHminimal}).

\begin{theoremalpha}\label{E} In the setting of Theorem \ref{Di}:

\begin{itemize}
\item[(a)] If $\omega_X$ is $a_X$-big, then $h^{0,j}(X)=h^{0,j}(Y)$ for all $j$. \\

\item[(b)] If $\omega_X$ is nef and $a_X$-big, then the rational Hodge structures of $X$ and $Y$ are isomorphic for all $k$. In particular, $h^{i,j}(X)=h^{i,j}(Y)$ for all $i$ and $j$. 
\end{itemize}
\end{theoremalpha}

\noindent{\bf Chen-Jiang decompositions. } In answer to a question of M. Popa, we prove the derived invariance of the Chen-Jiang decomposition of pushforwards of pluricanonical sheaves under the Albanese map (Proposition \ref{prop:pop}). This decomposition is a more refined version 
than the one appearing in Theorem \ref{C}(a), which is often very useful in the study of the birational geometry of irregular varieties (see e.g. \cite{cj}, \cite{loposc}, \cite{meng-popa}). 

Finally, variants and further applications are provided in Theorem \ref{thm:minimal} and in \S \ref{ultima}. The content of Appendix in \S\ref{rem:support} is  a property of independent interest concerning Chen-Jiang decompositions (used in the proof of Proposition \ref{prop:pop}).

Let us briefly indicate the methods of the proofs. Crucial background results are recalled in \S \ref{R0}. These include: (1) an explicit description of the Rouquier isomorphism  (\cite{ps1}), especially when applied to the non-vanishing loci of pluricanonical bundles (\cite{lombardi-derived}); \  (2) the relation, based on generic vanishing theory,  between the  non-vanishing loci of pluricanonical bundles  and the Albanese-Iitaka variety $\Alb Z_X$ (\cite{ch-iitaka1} and \cite{hps}); \ (3) basic properties  of the Fourier-Mukai-Poincar\'e (FMP) derived equivalence for an abelian variety and its dual. In some points we also make use of the recent symmetric version of such equivalence, due to Schnell (\cite{schnell}). 

 Theorem \ref{A} occupies \S \ref{F0}.  The essential ingredient is an adaptation of the classical results of Kawamata  on the geometric properties of the support of the kernel of a derived equivalence (\cite{ka2}). In doing this the above-mentioned preliminary results (1) and (2)  play a crucial role.  

The arguments  of Theorems \ref{B} and \ref{C} (\S \ref{H} and \S \ref{E0}) consist of several steps, hopefully of independent interest. In particular, we make essential use of the FMP derived equivalence. The main step is Theorem \ref{thm:rel-direct-image}, which  in turn makes use of Theorem \ref{thm:invpoi}. Briefly, we first prove the isomorphisms between  the \emph{FMP transforms} of  the two sides of the isomorphisms of Theorems \ref{B} and \ref{C}, and then apply the inverse FMP transform in order to obtain the desired isomorphisms. The  same principle is at the base of the proof of the multiplicativity of the isomorphism of Theorem \ref{B}(a). An antecedent of this argument, interesting for its own sake, is the derived invariance of the \emph{paracanonical ring}, proved in the section of background results (Theorem \ref{para:theorem}).  We refer to the introduction of \S\ref{H} for a more detailed presentation of the line of the argument. 
Another essential ingredient in the proof of Theorem \ref{C}(a) is a direct-sum decomposition of direct images  of pluricanonical bundles under the Albanese map (see Proposition \ref{prop:simplified}). This in turn is proved using the previously mentioned Chen-Jiang decomposition (\cite{cj}, \cite{loposc}), a further result of generic vanishing theory. 

The applications provided by Theorems \ref{Di} and \ref{E}  are proved in \S \ref{E1} and \S \ref{lastsection0}.  They make essential use of Theorem \ref{C} and, in particular, of the scheme
\[
\Proj (\mathcal \U_X)\rightarrow \Alb Z_X
\]
where $\U_X$ is the sheaf of rings appearing in the statement. 
 The argument of Theorem \ref{Di}(b)  first proves that Theorem \ref{Di}(a) implies the equality of the \emph{Gorenstein volumes} of $X^{can}_{a}$ and $Y^{can}_{a}$ in the ring $\widehat{\mathcal{M}_{\mathbb{C}}^{\prime}}$. In order to pass from the Gorenstein volumes to the classes of the varieties $X$ and $Y$ themselves, we make use of well-known results of the minimal model program.

\noindent \textbf{Acknowledgements.}
We are grateful to  D. Huybrechts, Z. Jiang, M. Popa, Ch. Schnell and P. Stellari for
beneficial conversations. FC thanks F. Bernasconi for a useful discussion after a talk  at EPFL.

\section{Background material}\label{R0}
 
In this section we establish the notation and   recall known results which 
will be used throughout the paper. 
The only new result is Theorem \ref{para:theorem} regarding  
the derived invariance of the paracanonical ring.

\subsection{Derived equivalences.} \label{D} 
A variety is a smooth and irreducible projective variety defined over 
 $\mathbb C$, unless otherwise specified.  
 If $X$ is a variety, 
 we denote by $D(X) = D^b(\mathcal{C}oh (X))$ its 
 bounded derived category of coherent sheaves, 
 endowed  with its natural structure of triangulated category. 
 
Let $X$ and $Y$ be varieties. An exact equivalence  $\Phi \colon D(X) \rightarrow D(Y)$ 
will be referred to as  a \emph{derived equivalence}, or simply,   \emph{equivalence}.
By Orlov's representability theorem, any  equivalence  is a Fourier-Mukai functor. 
Namely there exists a complex $\E$ in  
$D(X\times Y)$, unique up to isomorphism, such that 
$\Phi$ is isomorphic to the functor
 $$\Phi_{\E}(-)\; := \; {\bf R} q_* \big( p^* (-) \stackrel{\bf L}{\otimes} \E \big),$$ where $p$ and $q$ are the projections from $X\times Y$ onto
$X$ and $Y$, respectively. 
A quasi-inverse of $\Phi_{\E}$ is given by the  equivalence 
$\Phi_{ \ad \E }$ associated to the 
kernel
$$ \ad  \E  \;  := \;  \E^{\vee} \otimes p^*\omega_X[n] \; \cong \; \E^{\vee} \otimes q^*\omega_Y[n], $$
where $n=\dim X=\dim Y$ (recall that the dimension of a variety is 
a derived invariant),   $\omega_X$ (\emph{resp}.\ $\omega_Y$) denotes the canonical line bundle of $X$ (\emph{resp}.\ $Y$), and
$\E^{\vee} := \mathbf R \mathcal H om( \E , \OO_{X\times Y})$ is the derived dual.
Denote now by $$\delta_X \colon X \hookrightarrow X \times X
\quad \mbox{ and } \quad \delta_Y \colon Y \hookrightarrow Y\times Y$$ the diagonal embeddings of $X$ and $Y$, respectively.
Moreover, denote  by 
$p_{ij}$ the projections from the product  $X\times X\times Y\times Y$ onto the $i$-th and $j$-th factor, respectively.
 By a further  result of Orlov \cite[p. 535]{orlov},
the object $$\E \boxtimes \ad \E := 
p_{13}^*\E \otimes p_{24}^* \ad \E  \quad \mbox{ in } \quad D(X\times X \times Y \times Y) $$ defines a new  equivalence 
$\Phi_{\E\boxtimes \ad \E} \colon D(X\times X) \to D(Y\times Y)$ which satisfies  
\begin{equation}\label{orlov-classical}
\Phi_{\E \boxtimes \ad \E } \big( \delta_{X*}  \omega_X^{\otimes m}  \big) \; \cong  \; \delta_{Y*} \omega_Y^{\otimes m}
\quad \quad \forall \,  m\in \ZZ.
\end{equation}

\subsection{The Rouquier isomorphism and Rouquier-stable line bundles} \label{R1} 
 We will adopt the following notation.
 After having fixed Poincar\'e line bundles $\cP$ and $\Q$  on $X\times \Pic0 X$ and $Y\times \Pic0 Y$, respectively,
  we denote by $P_\alpha := \cP|_{  X \times \{ \alpha \} }$ (\emph{resp.} $Q_\beta:=\Q|_{ Y\times \{\beta\}}$)
 the corresponding line bundle on $X$ (\emph{resp.} $Y$) parametrized by the point $\alpha \in \Pic0 X$ 
 (\emph{resp.} $\beta\in \Pic0 Y$).
 Moreover, given an automorphism $f \colon X \rightarrow X$, we denote by $\gamma_f \colon  X \hookrightarrow X\times X$ 
 the  graph-embedding $x \mapsto(x , f(x))$. 
 With this terminology, we point out that $\gamma_{{\rm id}_X} = \delta_X$ is the diagonal embedding.
 
In  \cite{rouquier} (see also \cite{ps1} for the present formulation), Rouquier shows that
the equivalence $\Phi_\E \colon D(X) \to D(Y)$  induces a functorial   isomorphism of algebraic groups 
\begin{equation}\label{rouquier}
\varphi_\E \colon  \Aut0 X \times \Pic0 X \stackrel{ \sim }{\longrightarrow} \Aut0 Y\times \Pic0 Y
\end{equation}
which will be referred to as the  \emph{Rouquier isomorphism}. 
As pointed out by Popa and Schnell in \cite[Lemma 3.1]{ps1}, 
 this can be explicitly described as follows:  given pairs 
 $(f,\alpha ) \in \Aut0 X \times \Pic0 X$ and $(h , \beta) \in \Aut0 Y \times \Pic0 Y$, then 
$\varphi_\E ( f , \alpha ) = ( h, \beta )$ 
 if and only if
 \begin{equation}\label{ps} 
 p^*P_\alpha \otimes (f \times \, {\rm id}_Y)^*\E \; \cong \; q^*Q_\beta\otimes ( {\rm id}_X\times h )_*\E.
 \end{equation}

 By a result of the second author   \cite[Lemma 2.1]{lombardi-derived}, the following generalization of 
  \eqref{orlov-classical} holds: if  $\varphi_\E( f , \alpha)=( h , \beta)$, then
 \begin{equation}\label{lombardi1}
 \Phi_{\E\, \boxtimes \,\ad\E} \bigl( { \gamma_f }_* ( \omega_X^{\otimes m} \otimes P_\alpha ) \bigr) \; \cong 
 \;  { \gamma_h }_* ( \omega_Y^{\otimes m} \otimes Q_\beta) \end{equation}
  for all $m\in \mathbb Z$.

  A closed point $\alpha \in \Pic0 X$, or, equivalently, the corresponding line bundle $P_\alpha$, is said to be
   \emph{Rouquier-stable with respect to the equivalence $\Phi_\E$} (or simply $R$-stable if this will not cause any confusion) if 
  \[\varphi_\E(\id_X,\alpha) \; = \; (\id_Y,\beta)\]
  for some $\beta\in\Pic0 Y$.
 If this is the case, with a slight abuse of notation we will set 
 \begin{equation}\label{eq:notation}
 \beta \; = \; \varphi_\E(\alpha).
 \end{equation}
 Hence, for $R$-stable line bundles,
 \eqref{lombardi1} reduces to
 \begin{equation}\label{unmixed}
 \Phi_{\E\, \boxtimes \,\ad\E} \bigl( { \delta_X }_* ( \omega_X^{\otimes m} \otimes P_\alpha ) \bigr) \; \cong 
 \;  { \delta_Y }_* ( \omega_Y^{\otimes m} \otimes Q_{\varphi_\E(\alpha)}). \end{equation}

\subsection{Non-vanishing loci }\label{R2}

Given a coherent  sheaf $\F$ on $X$,   the following Zariski-closed subset  of $\Pic0 X$
\begin{equation}\label{GL-set}
V^0 ( X , \F ) \; = \; V^0(\F) \; = \;  \{ \, \alpha\in \Pic0 X \> | \> h^0(X,\F\otimes P_\alpha) > 0 \, \}
\end{equation}
will be referred to as \emph{the non-vanishing locus of $\F$.} Note that 
\begin{equation*}
V^0(X,\F) \; = \; V^0(\Alb X, {a_X}_*\F)\>.
\end{equation*}

 In this paper the non-vanishing loci $V^0(X,\omega_X^{\otimes m})$ will play an important role. 
 By  seminal results in generic vanishing theory, they are finite unions of torsion translates of abelian subvarieties of $\Pic0 X$. A more precise description will  be given in \S\ref{R3} below. Their importance in the present paper stems from another result of the second author (\cite[Proposition 3.1]{lombardi-derived}) stating, in the terminology of \S\ref{R1},  that: \emph{if $\alpha\in V^0(X,\omega^{\otimes m}_X)$  for some $m\in \mathbb Z$, then $\alpha$ is $R$-stable, and
  the equivalence $\Phi_{\E\boxtimes \ad\E}$ induces an isomorphism
\begin{equation}\label{eq:isoh0v0}
H^0(X , \omega_X^{\otimes m} \otimes P_{\alpha})  \; \cong \; H^0(Y, \omega_Y^{\otimes m} \otimes Q_{\varphi_\E(\alpha)}) \end{equation}
for all $m\in\mathbb Z$.}

In particular, it follows that 
the Rouquier isomorphism acts  as
 \begin{equation}\label{V0}
 \varphi_{\E} \bigl( {\rm id}_X ,  V^0 (X , \omega_X^{\otimes m} ) \bigr) = \bigl( {\rm id}_Y ,  V^0(Y , \omega_Y^{\otimes m} ) \bigr) 
  \end{equation}
for all $m\in \mathbb Z$.
 Hence, for any $m\in \mathbb Z$,  the Rouquier isomorphism induces 
 an isomorphism of algebraic closed sets $V^0(X, \omega_X^{\otimes m} ) \cong  V^0 ( Y , \omega_Y^{\otimes m} )$.

\subsection{The paracanonical ring. } \label{para:subsection} The \emph{paracanonical ring} of a smooth projective variety is, by definition, the bigraded ring
\[ P(X,\omega_X) \; = \; \bigoplus_{m\geq 0, \>\alpha\in \Pic0 X} H^0(X,\omega_X^{\otimes m}\otimes P_\alpha).\]
We prove that this ring is a derived invariant.

\begin{theorem}\label{para:theorem} Given 
an equivalence $\Phi_\E \colon D(X)\rightarrow D(Y)$ of smooth projective varieties,
the equivalence $\Phi_{\E\boxtimes\ad \E}$ induces an isomorphism of algebras
\[P(X,\omega_X) \; \cong \; P(Y,\omega_Y).\]
\end{theorem}
\begin{proof} The isomorphism $P(X,\omega_X)\cong P(Y,\omega_Y)$ 
follows from \eqref{eq:isoh0v0} and \eqref{V0} because the $(m,\alpha)$-graded component of the paracanonical ring is non-zero if and only if $\alpha\in V^0(X,\omega_X^{\otimes m})$. As usual, the multiplicativity holds because the multiplication maps 
\begin{equation}\label{eq:twisted}
H^0(X,\omega_X^{\otimes m}\otimes P_\alpha)\otimes H^0(X,\omega_X^{\otimes k}\otimes P_{\alpha'})\longrightarrow 
H^0(X,\omega_X^{\otimes (m+k)}\otimes P_{\alpha+\alpha'})
\end{equation}
 are given by composition (see e.g. \cite[p. 535]{orlov} or \cite[Proposition 6.1]{huybrechts}) for all $\alpha, \alpha' \in \Pic0 X$ and $m,k\in \mathbb Z$.
\end{proof}
Note that that the canonical ring is a subring of the paracanonical ring. Hence the above theorem includes Orlov's derived invariance of the canonical ring (\cite[Corollary 2.1.9]{orlov}).

\subsection{The  Albanese-Iitaka morphism}\label{R3}  
 Let $X$ be  a variety with non-negative Kodaira dimension. In this subsection we recall work of Chen, Hacon, Pardini, Popa and Schnell, among others,
  showing that for $m\ge 1$ the geometry of the  non-vanishing loci $V^0(X,\omega_X^{\otimes m})$ 
is closely related to the geometry of the  Iitaka fibration. 

To set up the notation, let us choose a smooth birational modification $\widetilde{X} \rightarrow X$ such that the Iitaka fibration is represented by a morphism 
$f_{\widetilde X} \colon \widetilde{X} \rightarrow Z_X$ with $Z_X$ smooth. 
There is a  commutative diagram
\begin{equation}\label{diag}
\xymatrix{
\widetilde X \ar@/^2pc/[rr]^{a_{\widetilde X}} \ar[r] \ar[rd]_-{f_{\widetilde X} } & X \ar[r]^-{a_X} \ar@{-->}[d] \ar[dr]^-{c_X} & \Alb ( X) \ar[d]^-{p_X} \\
& {Z_X} \ar[r]^-{a_{Z_X}} & \Alb (Z_X)
}
\end{equation}
where $a_X, a_{\widetilde X}$ and $a_{Z_X}$ are  Albanese maps, and  $p_X$ is a fibration (\emph{i.e.}  a surjective morphism
 with connected fibers) of abelian varieties 
induced by $f_{\widetilde X}$ (see \cite[Proposition 2.1.b]{hp}\footnote{in the quoted reference this is stated only for varieties of maximal Albanese dimension, but the proof works in general}). 
Note that the abelian variety $\Alb Z_X$ and the morphisms $p_X$ and  $c_X = p_X \circ a_X$  only depend on $X$,
 and not on the modification $\widetilde X$. 
We call 
\[c_X \colon X\rightarrow \Alb Z_X\]
 the \emph{Albanese-Iitaka morphism} of $X$. 
 We recall that the Albanese morphism $a_X$ is defined up to translation in $\Alb X$, and similarly the morphism $c_X$ is defined up to translation in $\Alb Z_X$. The choice of a morphism $a_X$ determines a morphism $c_X$. 

As previously mentioned, the irreducible components of the loci $V^0(X,\omega_X^{\otimes m})$ are torsion translates 
of abelian subvarieties of $\Pic0 X$. For $m=1$ this is part of a classical result of Green-Lazarsfeld (\cite{gl2}), 
while for $m\geq 2$ this is  a result of Chen-Hacon
  (\cite[Theorem 3.2]{chenhaconirr}) see also \cite[Theorem 3.5]{lai} and  \cite[Theorem 10.1]{hps}). Moreover, by another  result of Chen-Hacon  (\cite[Lemma 2.2]{ch-iitaka1}; see also  \cite[(2) after Lemma 11.1]{hps}) for $m=1$ such subvarieties  are contained in torsion translates of  the abelian subvariety $p_X^*\Pic0 Z_X$, namely
\begin{equation}\label{eq:non-van1}
V^0(\Alb X,{a_X}_*\omega_X) \; \subseteq \; \bigsqcup_{i=1}^{N_X(1)} p_X^*\Pic0 Z_X-\alpha_{1,i}\>
\end{equation}
where 
$\alpha_{1,i}$ are torsion points of $\Pic0 X$ and the translates $p_X^*\Pic0 Z_X-\alpha_{1,i}$ are distinct  translates
intersecting  $V^0(\Alb X, {a_X}_*\omega_X)$ non-trivially. 
For $m>1$, by a theorem of Hacon-Popa-Schnell (\cite[Theorem 11.2(b)]{hps}), 
the above mentioned abelian subvarieties \emph{coincide} with $p_X^*\Pic0 Z_X$, namely
 \begin{equation}\label{eq:non-van2}
V^0(\Alb X,{a_X}_*\omega_X^{\otimes m}) \; = \; \bigsqcup_{i=1}^{N_X(m)} p_X^*\Pic0 Z_X-\alpha_{m,i}\>.
\end{equation}

Let $G_X$ be the union of the loci
$V^0(X,\omega_X^{\otimes m})$ for all $m\ge 0$.  As observed in  \cite[Lemma 3.3]{chenhaconirr}, $G_X$ is  a subgroup of $\Pic0 X$. Indeed, it is a semigroup because if $\alpha_1\in V^0(X,\omega_X^{\otimes m})$ and $\alpha_2\in V^0(X,\omega_X^{\otimes n})$ then clearly $\alpha_1+\alpha_2\in V^0(X,\omega_X^{\otimes (m+n)})$.  The fact that the various components are \emph{torsion} translates of abelian subvarieties make $G_X$ automatically a group, containing $\Pic0 Z_X$ as a subgroup. Although not strictly necessary for the results of the present paper, it is worth to remark that $\Pic0 Z_X$ has finite index in $G_X$. This was proved, using generic vanishing methods, in \cite[p.\ 204 and Corollary 3.4]{chenhaconirr}, and follows also by the general finite generation result of \cite {bchm} (see Remark \ref{rem:finite-gen} below).

Now, let $Y$ be another smooth projective variety. There are similar decompositions on $Y$:
\begin{equation}\label{eq:non-van1-Y}
V^0(\Alb Y,{a_Y}_*\omega_Y)\; \subseteq \; \bigsqcup_{i=1}^{N_Y(1)} p_Y^*\Pic0 Z_Y-\beta_{1,i}\>
\end{equation}
 \begin{equation}\label{eq:non-van2-Y}
V^0(\Alb Y,{a_Y}_*\omega_Y^{\otimes m}) \; = \; \bigsqcup_{i=1}^{N_Y(m)} p_Y^*\Pic0 Z_Y-\beta_{m,i}
\end{equation}
for some torsion points $\beta_{m,i}\in \Pic0Y$.
From \eqref{V0}, it follows that
if  $X$ and $Y$ are derived equivalent, then  the Rouquier isomorphism induces  isomorphisms of algebraic groups
 \begin{equation}\label{eq:G}
 \varphi_\E \colon G_X \xrightarrow{\sim} G_Y,
 \end{equation}

 \begin{equation}\label{eq:PicZ}
\varphi_\E \colon \Pic0 Z_X \xrightarrow{\sim} \Pic0 Z_Y
 \end{equation}
 and, moreover, equalities
 \begin{equation}\label{eq:N(m)}
 N_X(m) \; =\; N_Y(m)
 \end{equation}
 for all $m\ge 1$. Finally, up to reordering, one can arrange 
 \[
 \beta_{m,i} \; = \; \varphi_\E(\alpha_{m,i})\>.
 \]
 Crucial to our main results will be a sheaf-theoretic version of the decompositions \eqref{eq:non-van1} and \eqref{eq:non-van2}. 
 This will  be the content of Proposition \ref{prop:simplified} in the sequel.

 \subsection{(Symmetric) FMP transform on abelian varieties and generic vanishing }\label{R4}
Given an abelian variety $A$, the  normalized Poincar\'e line bundle $\cP$ on $A\times \Pic0 A$ may be regarded as  the kernel of an equivalence, introduced in  \cite{mukai}:
 \begin{equation}\label{eq:FMP}
 \Phi_\cP \colon D(A)\rightarrow D(\Pic0 A). 
 \end{equation}
We will refer to it as the Fourier-Mukai-Poincar\'e (FMP) transform.  In the recent paper \cite{schnell}, Schnell introduced a useful variant of it, the \emph{symmetric FMP transform}, which carries the same information, but renders the notation and several arguments more transparent. It is defined as follows:
 \[ 
 \FM_A \; : = \; \Phi_{\cP}\circ\Delta_A \colon D(A)\rightarrow D(\Pic0 A)^{op}
\]
where $\Delta_A ( - )=\mathbf R\mathcal{H}om( - ,\OO_A[g])$ and $g=\dim A$.

In this language, a coherent  sheaf $\F$ on $A$ is said to be a \emph{GV-sheaf} (generic vanishing sheaf) if $\FM_A(\F)$ is a sheaf (in degree $0$). If this is the case,  the usual notation is
 \[
 \FM_A(\F) \;  \cong \; \reallywidehat{\F^\vee}\>.
 \] 
 If the sheaf $\reallywidehat{\F^\vee}$ is, in addition, torsion-free, then the sheaf $\F$ is said to be \emph{$M$-regular}. 
 References concerning these notions are for instance contained in 
   \cite{msri}, \cite{PP3}, \cite{paposc} and \cite{schnell}. It suffices here to say that, by base change and Serre duality, if $\F$ is a $GV$-sheaf, then the set-theoretic support of the sheaf  $\reallywidehat {F^\vee}$ is the subvariety -$V^0(A, \F)$, and the  fiber $\reallywidehat{\F^\vee}\otimes\mathbb C(\alpha)$ at a point $\alpha\in\Pic0 A$ is naturally identified  to the linear space $H^0(A,\F\otimes P_{-\alpha})^\vee$  (see \S \ref{rem:support} especially  (\ref{eq:scheme-supp}) below for the scheme-theoretic support). 
If $\F$ is a nonzero $M$-regular sheaf, then of course the support of 
$\reallywidehat{\F^\vee}$ is the
 full $\Pic0 A$ and 
\begin{equation}\label{eq:non-van}
V^0(A,\F) \; = \; \Pic0 A \>.
\end{equation}

The generic vanishing theorems of Hacon and Popa-Schnell for morphisms $f \colon X\rightarrow A$ from smooth complex projective varieties to abelian varieties assert that \emph{for all $m\ge 1$ the direct images $f_*\omega_X^{\otimes m}$ are $GV$-sheaves} 
(see Hacon's paper \cite[Theorem 1.5]{ha} for the case $m=1$, and Popa-Schnell's paper  \cite[Theorem 1.10]{ps3} for the other cases 
$m>1$). A refinement of these results (the \emph{Chen-Jiang decomposition}) will be recalled in the course of  the proof of Proposition \ref{prop:simplified}.


\section{Derived invariance of the Stein factorization of the Albanese-Iitaka morphism}\label{F0}

In this section  we  consider the  Stein factorization
of the Albanese-Iitaka morphism. 
Let  $X$ and $Y$ be  smooth projective varieties 
and consider the commutative diagrams
 \begin{equation}\label{xy}
 \xymatrix{ X \ar[r]^-{a_X} \ar[rd]^-{c_X} \ar[d]^-{s_X} & \Alb X \ar[d]^-{p_X} \\ 
 X' \ar[r]^-{c'_X}   & \Alb Z_X} \qquad \qquad
 \qquad \xymatrix{ Y \ar[r]^-{a_Y} \ar[rd]^-{c_Y} \ar[d]^-{s_Y} & \Alb Y \ar[d]^-{p_Y} \\ 
 Y' \ar[r]^-{c'_Y} & \Alb Z_Y} 
 \end{equation}
   where we keep the notation of \S\ref{R3}, and
	the left-bottom sides of the diagrams
are the Stein factorizations of $c_X$ and $c_Y$. 

In this section we prove that a derived equivalence $\Phi_\E\colon D(X)\rightarrow D(Y)$ induces an isomorphism of the bottom arrows of the above diagrams, 
\emph{i.e.}   of the finite morphisms appearing in  the Stein factorizations of the Albanese-Iitaka morphisms. 
 To introduce the precise statement,  given a morphism of abelian varieties $f \colon A\rightarrow B$, we denote  by $\widehat f \colon \Pic0 B \rightarrow \Pic0 A$ the dual morphism. By dualizing the induced Rouquier isomorphism \eqref{eq:PicZ}, we obtain the isomorphism $\widehat{\varphi_\E}\colon\Alb Z_Y\rightarrow \Alb Z_X$.

\begin{theorem}\label{step1} 
An equivalence $\Phi_\E\colon D(X)\rightarrow D(Y)$ induces
an isomorphism $\psi \colon Y^\prime\rightarrow X^\prime$  such that the following diagram 
\begin{equation}\label{xy2.0}
\xymatrix{X' \ar[d]^{c_X'} &Y' \ar[l]^{\psi}_\sim\ar[d]^{c_Y'}\\
\Alb Z_X & \Alb Z_Y \ar[l]^{\widehat{\varphi_\E}}_\sim}
\end{equation}
is commutative. 
\end{theorem}

\begin{remark}\label{rem:upto} By 
recalling that the  Albanese-Iitaka morphisms $c_X$ and $c_Y$ are defined only up to translations, the meaning of the above statement is that one can find representatives of $c_X^\prime$ and $c_Y^\prime$ such that diagram \eqref{xy2.0} is commutative.
\end{remark}

\begin{proof} 
We adapt an argument of \cite[Theorem 1]{lombardi-fibrations},   inspired by  Kawamata's kernel support technique \cite{ka2}, to our more general setting.
By keeping  the notation of  diagrams \eqref{xy}, we  consider the support 
  ${\rm Supp}(\E) = \bigcup_j \mathcal H^j(\E)$ of the kernel $\E$ (with the reduced structure). We will show that the image of  ${\rm Supp}(\E) $  via the morphism 
\[(s_X\times s_Y) \colon X \times Y\rightarrow X' \times Y' \]
is the graph of an isomorphism  $\psi \colon Y' \rightarrow X'$ satisfying the conclusion of the statement. 
 
Let $p'$ and $q'$ be the natural projections from $X' \times Y'$ onto the first and second factor, respectively. 
In the first place we note that:\\
  (*) \emph{ the restrictions of the projections $p'$ and $q'$ to $(s_X\times s_Y) \big( {\rm Supp}(\E) \big)$ dominate   $X'$ and $Y'$, respectively, and have connected fibers. }
 	
	\noindent This follows from the well-known fact that
${\rm Supp}(\E)$ itself dominates both $X$ and $Y$, via morphisms with connected fibers (\cite[Lemma 6.4]{huybrechts}). 

Next, we claim that the projection  from $(s_X\times s_Y) \big({\rm Supp}(\E) \big)$ onto $X'$ has finite fibers. 
 In order to prove this, firstly  we   note that, for a sufficiently ample  line bundle $L$ on $X\times Y$, there is an equality of closed subsets
\begin{equation}\label{exchange}
{\rm Supp}\big( \mathbf R(s_X\times s_Y)_*(\E \otimes L) \big) \; = \; (s_X\times s_Y) \big( {\rm Supp} \big(\E) \big).
\end{equation}
Indeed, since $L$ is sufficiently positive, by the degeneration of the hypercohomology spectral sequence computing $R^i(s_X\times s_Y)_*(\E \otimes L)$, we have that 
\begin{equation}\label{eq:hyper}
R^i  ( s_X\times s_Y)_*(\E\otimes L) \;  \cong \;  (s_X\times s_Y)_*\mathcal H^i( \E\otimes L)\>
\end{equation}
 This proves \eqref{exchange}. 
Now, every line bundle on $X$ parametrized by a point $\alpha \in c_X^*\Pic0 Z_X$ is of the form $c_X^*P^\prime_\alpha$, 
where $P^\prime_\alpha$ is a line bundle on $\Alb Z_X$. Similarly, we denote $Q^\prime_\beta$  line bundles parametrized by $\Pic0 Z_Y$.
Thus the relation \eqref{ps} defining the Rouquier isomorphism $\varphi_\E$ can be written as follows
\[ 
p^* c_X^* P^\prime_\alpha \otimes \E \;  \cong \; q^* c_Y^* Q^\prime_{ \varphi_\E( \alpha ) } \otimes \E .
\] 
Tensorizing with our sufficiently positive line bundle $L$ on $X\times Y$ we get 
\begin{equation}\label{isom}
p^* c_X^* P^\prime_\alpha  \otimes \E \otimes L \;  \cong \; q^* c_Y^*Q^\prime_{ \varphi_\E( \alpha ) }  \otimes \E \otimes L.
\end{equation} 
Applying $\mathbf R(s_X\times s_Y)_*$, by projection formula we get
\begin{equation}\label{isom2}
{p'}^*{c'_X}^*P^\prime_\alpha \otimes \mathbf R (s_X\times s_Y)_*(\E\otimes L) \; \cong \; q'^*{c'_Y}^*Q^\prime_{\varphi_\E(\alpha)}
\otimes  \mathbf R (s_X \times s_Y)_*(\E\otimes L).
\end{equation}
Assume that the projection $p'$  contracts a curve $C\subset {\rm Supp}\big( \mathbf R(s_X\times s_Y)_*(\E\otimes L) \big)$. 
Let $i$ be an integer such that the irreducible curve $C$ is contained in ${\rm Supp}\big( R^i(s_X\times s_Y)_*(\E\otimes L) \big)$.
Taking the $i$-th cohomology $\mathcal{H}^i$ in  \eqref{isom2} and then
restricting  to $C$, it follows that
\[
 \big(R^i(s_X\times s_Y)_*(\E\otimes L)\big)|_{C} \; \cong \; {q^\prime}^*{c^\prime_Y}^*Q^\prime_{\beta}\otimes  \big(R^i(s_X\times s_Y)_*(\E\otimes L)\big)|_{ C}
\] 
for all $\beta = \varphi_{\E}(\alpha) \in \Pic0 Z_Y$. 
This implies that $({c'_Y}^*Q^\prime_{\beta})_{|C}$ is of finite order for all $\beta\in \Pic0 Z_Y$.
 In turn, this implies that $C$ is contracted by $c'_Y$. But  $c'_Y$ is a map with finite fibers.  Hence no curve can be contracted by the projection $p^\prime$, as claimed.

It follows in particular that $\dim X' \ge \dim Y'$. Since the role of $X$ and $Y$ can be exchanged we have that $\dim X' = \dim Y'$. 
Moreover, from $(*)$ it follows that $(s_X \times s_Y) \big({\rm Supp}(\E) \big)$  projects isomorphically onto both $X'$ and $Y'$, 
so that the map 
\[\psi := p'\circ q'^{-1} \colon Y' \to X'\]
 is  an isomorphism.

In order to prove the last part of the statement, it is enough to show that the isomorphism $\psi \colon Y^\prime \rightarrow X^\prime$ induces an isomorphism $\psi^* \colon {c_X'}^*\Pic0 Z_X\rightarrow {c_Y'}^*\Pic0 Z_Y$ which agrees with the Rouquier isomorphism $\varphi_\E$ of \eqref{eq:PicZ}. By restricting \eqref{isom2} to   the normal variety ${\rm Graph}(\psi) $, and by  taking determinants,  we find a positive  integer $r$ such that
$$\psi^*  {c^\prime_X}^*(P^\prime_{\alpha})^{\otimes r}  \, \cong \,  
{c_Y^\prime}^*(Q^\prime_{\varphi_{\E}(\alpha) })^{\otimes r}$$ for all  $\alpha\in \Pic0 Z_X$. 
From this we deduce that, for all $\alpha\in \Pic0 Z_X$,  
\[
\psi^*   {c_X^\prime}^*P^\prime_{\alpha} \; \cong \; {c_Y^\prime}^*Q^\prime_{\varphi_{\E}(\alpha) } \>.
\]
This is exactly what needed.
\end{proof}

\section{Relative canonical ring and Hochschild structure of the Albanese-Iitaka morphism}\label{H}

This section is concerned with the derived invariance  of the relative canonical ring under the Albanese-Iitaka morphism (Theorem \ref{thm:invmult}). This will be crucial for the result concerning  the relative canonical ring under  the Albanese morphism (Theorem \ref{thmbasechange}). Along the way, in Subsection \ref{HHoch}, we will  consider a relative version of the algebras in  \cite[Corollary 2.1.10]{orlov} and show its derived invariance  (as graded coherent sheaf, see Corollary \ref{cor:hoch1}).  

The argument involves several steps as follows: we first globalize the isomorphisms (\ref{unmixed}) to an isomorphism involving the Poincar\'e line bundles (Theorem \ref{thm:invpoi} below). This allows the use of the FMP transform, yielding the isomorphisms (\ref{eq:rel-FM}) for FMP transforms, globalizing and generalizing the  isomorphisms of linear spaces (\ref{eq:isoh0v0}). Applying the inverse  FMP transform we get the desired isomorphisms for the relative canonical rings, and, more generally, for the previously mentioned relative algebras  (Theorem \ref{thm:rel-direct-image} and Corollaries \ref{cor:hoch1}, \ref{cor:homology}). Finally, we prove the isomorphism of the multiplicative structure of the relative canonical rings by reducing it to the isomorphism of the paracanonical rings,  again using the FMP transform (Theorem \ref{thm:invmult}). 

\subsection{Preliminary lemma.}\label{HLem} We will need the following preliminary standard Lemma. Let $R,T, A,B$ be smooth projective varieties. Let $\F$ be an object in $D(R\times T)$  and let $f\colon A\rightarrow B$ be an isomorphism. We consider the functor 
\[\Phi_\F\boxtimes f_*=\Phi_{\F\boxtimes \OO_{\Gamma_f}}\colon D(R\times A)\rightarrow D(T\times B)\>,\]
where $\Gamma_f\subset A\times B$ denotes the graph of $f$. 
We denote by $p_A \colon R\times A \to A$ and $p_B \colon T\times B \to B$  the natural projections.

\begin{lemma}\label{lem:prelim} Let $U\in D(R\times A)$ and $V\in D(A)$. Then there is an isomorphism
\[(\Phi_\F\boxtimes f_*)(U\stackrel{\bf L}{\otimes} p_A^* V) \; \cong \; \bigl((\Phi_\F\boxtimes f_*)(U)\bigr)\stackrel{\bf L}{\otimes}p_B^*( f_*V)\>.\]
\end{lemma} 
\begin{proof} From the basic relation on $A\times B$ 
\begin{equation}\label{eq:intermediate0} 
q_A^*V\stackrel{\mathbf L}{\otimes}\OO_{\Gamma_f} \; \cong\; q_B^*(f_*V)\lotimes \OO_{\Gamma_f}\>,
\end{equation}
we get the relation on $(R\times A)\times (T\times B)$
\begin{equation}\label{eq:intermediate}\bar p_A^{\,*}V\stackrel{\mathbf L}{\otimes}p_{A\times B}^*\OO_{\Gamma_f} \; \cong \; \bar p_B^{\, *}(f_*V)\lotimes p_{A\times B}^*\OO_{\Gamma_f}\>,
\end{equation}
where in \eqref{eq:intermediate0} $q_A$ and $q_B$ are the projections from $A\times B$ onto $A$ and $B$, respectively, while 
$\bar p_A$ and $\bar p_B$ in \eqref{eq:intermediate}  are the projections from $R\times A\times T\times B$ onto $A$ and $B$, respectively.
Therefore there are isomorphisms
\begin{eqnarray*}\Phi_{\F\boxtimes \OO_{\Gamma_f}}(U\lotimes p_A^*V)&\cong&{\mathbf R}p_{T\times B\,*}\bigl(p_{R\times A}^*(U\lotimes p_A^*V)\lotimes (\F\boxtimes\OO_{\Gamma_f})\bigr)
\\&\cong &{\mathbf R}p_{T\times B\,*}\bigl(p_{R\times A}^*U\lotimes \bar p_A^{\, *}V\lotimes (\F\boxtimes\OO_{\Gamma_f})\bigr)
\\&\buildrel{\eqref{eq:intermediate}}\over\cong&{\mathbf R}p_{T\times B\,*}\bigl(p_{R\times A}^*U\lotimes \bar p_B^{\, *}(f_*V)\lotimes (\F\boxtimes\OO_{\Gamma_f})\bigr)\\
&\cong&
\bigl((\Phi_\F\boxtimes f_*)(U)\bigr)\stackrel{\bf L}{\otimes}p_B^*( f_*V)
\end{eqnarray*}
where the last isomorphism follows by the projection formula.
\end{proof}

\subsection{Setting and notation. }\label{HNot} 
The next step will be the proof of Theorem \ref{thm:invpoi} here below, which is in fact a global version of \eqref{unmixed}.  As usual we assume that there is an equivalence $\Phi_\E \colon D(X)\rightarrow D(Y)$. To establish the notation, recall that the Rouquier isomorphism induces an isomorphism $\varphi_\E \colon \Pic0 Z_X\rightarrow \Pic0 Z_Y$ (see \eqref{eq:PicZ}). 
We choose Poincar\'e line bundles $\cP_Z$ and $\Q_Z$ respectively on $\Alb Z_X\times \Pic0 Z_X$ and $\Alb Z_Y\times \Pic0 Z_Y$, normalized so that $(\cP_Z)|_{\{e_X\}\times \Pic0 Z_X}$ is trivial where $e_X$ is the identity element of $\Alb Z_X$, 
and similarly for $\Q_Z$. From the universal property of the Poincar\'e line bundle, 
it follows that $(\widehat{\varphi_{\E}}^{-1}\times\varphi_\E)^*\Q_Z$ is a Poincar\'e line bundle on $\Alb Z_X$. Since it satisfies the above normalization, it follows that
\[
(\widehat{\varphi_{\E}}^{-1}\times \varphi_\E)^*\Q_Z \; \cong \; \cP_Z\>.
\]

Let 
\[\cP_{Z,X} \; := \; (c_X \times {\rm id}_{\Pic0 Z_X})^*\cP_Z \qquad \mbox{and} \qquad \Q_{Z,Y}\; := \; (c_Y \times {\rm id}_{\Pic0 Z_Y})^*\Q_Z\]
be the induced Poincar\'e line bundles on $X \times \Pic0 Z_X$ and $Y \times \Pic0 Z_Y$, respectively.
We consider the fibered diagonal embedding
\begin{equation}\label{eq:reldiag}
\widetilde{\delta_X} = ( \delta_X \times {\rm id}_{\Pic0 Z_X}) \colon X\times\Pic0 Z_X\rightarrow X\times X\times \Pic0 Z_X, \qquad(x,\alpha)\mapsto (x,x,\alpha)
\end{equation}
and similarly for $Y$.
Finally, we consider the equivalence 
\begin{equation}\label{eq:relequiv}  
\Phi_{\E\, \boxtimes \,\ad\E}\boxtimes {\varphi_\E}_* = 
\Phi_{\E\, \boxtimes \,\ad\E \boxtimes \OO_{\Gamma_{\varphi_{\E}}}}
\colon D(X\times X\times \Pic0 Z_X)\rightarrow D(Y\times Y\times \Pic0 Z_Y)
\end{equation}
where we are adopting the notation of \S\S \ref{D}, \ref{R1} and \ref{R3}.
We also denote by 
$q_X \colon X \times \Pic0 Z_X \to X$ and $q_Y \colon Y \times \Pic0 Z_Y \to Y$ the natural projections onto the first factors.


\subsection{Global version of \eqref{unmixed}.  }\label{Hpluri}

\begin{theorem}\label{thm:invpoi} 
In the setting of \S\ref{HNot} there are isomorphisms  
\[(\Phi_{\E\, \boxtimes \,\ad\E}\boxtimes {\varphi_\E}_*) \big({\widetilde{\delta_X}}_*(q_X^*\omega_X^{\otimes m}\otimes \cP_{Z,X}) \big) \; \cong\; {\widetilde{\delta_Y}}_*(q_Y^*\omega_Y^{\otimes m}\otimes \Q_{Z,Y})\]
for all $m\in \mathbb Z$.	
\end{theorem}
\begin{proof}  
As usual, given a closed point $\alpha\in \Pic0 Z_X$, we denote by $P_{Z,\alpha}$ the corresponding line bundle on $X$ 
(\emph{i.e.}   $c_X^*\big( \cP_{Z} |_{\Alb Z_X\times\{\alpha\}} \big))$, and we will adopt the notation $Q_{Z,\beta}$  for line bundles on $Y$ parametrized by points in $\Pic0 Z_Y$.  To begin with, we claim that, for all closed points $\beta$ of $\Pic0 Z_Y$, there are isomorphisms
\begin{equation}\label{eq:lombardi-p}
(\Phi_{\E\, \boxtimes \,\ad\E}\boxtimes {\varphi_\E}_*) \big( {\widetilde{\delta_X}}_*(q_X^*\omega_X^{\otimes m}\otimes \cP_{Z,X}) \big) 
\otimes \OO_{Y\times Y\times\{\beta\}} \; \cong\; {\delta_Y}_*(\omega_Y^{\otimes m}\otimes Q_{Z,\beta}).
\end{equation}
Indeed, for $\alpha=\varphi_\E^{-1}(\beta)$, we have that
\begin{eqnarray*}
\Phi_{\E\boxtimes \mathrm{ad}\E} \big( {\delta_X}_*(\omega_X^{\otimes m}\otimes P_{Z, \alpha}) \big)&\cong &
(\Phi_{\E\, \boxtimes \,\ad\E}\boxtimes {\varphi_\E}_*) \big( {\widetilde{\delta_X}}_*(q_X^*\omega_X^{\otimes m}\otimes \cP_{Z,X})\otimes \OO_{X\times X\times\{\alpha\}} \big)\\
&\cong & 
(\Phi_{\E\, \boxtimes \,\ad\E}\boxtimes {\varphi_\E}_*) \big({\widetilde{\delta_X}}_*(q_X^*\omega_X^{\otimes m}\otimes \cP_{Z,X}) \big) \lotimes \OO_{Y\times Y\times\{\beta\}}
\end{eqnarray*}
where the last isomorphism follows from Lemma \ref{lem:prelim} (note that, because of the flatness of ${\widetilde{\delta_X}}_*(q_X^*\omega_X^{\otimes m}\otimes \cP_{Z,X})$, the second underived tensor product in the right-hand side of the first line coincides with the derived tensor product). Therefore,  from the above quoted \eqref{unmixed}, it follows that
\[
(\Phi_{\E\, \boxtimes \,\ad\E}\boxtimes {\varphi_\E}_*) \big({\widetilde{\delta_X}}_*(q_X^*\omega_X^{\otimes m}\otimes \cP_{Z,X}) \big) \lotimes \OO_{Y\times Y\times\{\beta\}} \; \cong \; {\delta_Y}_*(\omega_Y^{\otimes m}\otimes Q_{Z, \beta})\> .
\]
Note that this implies that the derived tensor product in the left-hand side is in fact concentrated in degree zero. In conclusion,  we have proved \eqref{eq:lombardi-p}.

By the see-saw principle,  \eqref{eq:lombardi-p} implies that 
\begin{equation}\label{eq:inverse}
(\Phi_{\E\, \boxtimes \,\ad\E}\boxtimes {\varphi_\E}_*) \big( {\widetilde{\delta_X}}_*(q_X^*\omega_X^{\otimes m}\otimes \cP_{Z,X}) \big)
\; \cong \; {\widetilde{\delta_Y}}_*(q_Y^*\omega_Y^{\otimes m}\otimes\Q_{Z,Y})\otimes p_{\Pic0 Z_Y}^*L
\end{equation}
where $L$ is a line bundle on $\Pic0 Z_Y$ and 
$p_{\Pic0 Z_Y} \colon Y\times Y\times \Pic0 Z_Y \to \Pic0 Z_Y$ is the projection onto the third factor. 

Let $\Delta_Y$ be the diagonal in $Y\times Y$.
Finally, we prove that the line bundle $L$ is trivial by proving that the restriction to $\{(y_0,y_0)\}\times \Pic0 Z_Y$ of the line bundle on $\Delta_Y\times\Pic0 Y$ appearing in the left-hand side of \eqref{eq:inverse}  is the trivial line bundle, where $y_0\in Y$ is a point such that $c_Y(y_0)=e_Y$ (the identity 
element of $\Alb Z_Y$). Such restricted line bundle is the complex (concentrated in degree $0$)
\begin{equation}
(\Phi_{\E\, \boxtimes \,\ad\E}\boxtimes {\varphi_\E}_*) \big({\widetilde{\delta_X}}_*(q_X^*\omega_X^{\otimes m}\otimes \cP_{Z,X}) \big) 
\lotimes \OO_{\{(y_0,y_0)\}\times\Pic0 Z_Y}\>.
\end{equation}
By the projection formula, such line bundle can be described as 
\begin{equation}\label{eq:pformula}
 \R\widetilde q_*\bigl({\widetilde p}^* ( {\widetilde{\delta_X}}_*(q_X^*\omega_X^{\otimes m}\otimes \cP_{Z,X} ))\lotimes (\E\,\boxtimes\, \ad\E\boxtimes\OO_{\Gamma_{\varphi_\E}})\lotimes \OO_{X\times X\times \Pic0 Z_X \times \{(y_0,y_0)\}\times\Pic0 Z_Y}\bigr),
\end{equation}
where $\widetilde p$ and $\widetilde q$ are the projections from $(X\times X\times\Pic0 Z_X)\times (Y\times Y\times \Pic0 Z_Y)$
onto the first and second factor, respectively. 
By means of Theorem \ref{step1} one can prove 
that $(\E \lotimes \ad\E) \lotimes \OO_{X \times \{y_0\}}$ 
is set-theoretically supported on the fiber of $c_X$ over the identity point  $e_X \in \Alb Z_X$.\footnote{This goes as follows. 
By construction $(s_X \times s_Y)(\Supp(\E)) = {\rm Graph}(\psi)$ (see the proof of Theorem \ref{step1}), hence $\Supp(\E)$ is contained in $(s_X \times s_Y)^{-1}({\rm Graph}(\psi))$. So 
\[
\Supp( \E \lotimes \ad\E \lotimes \OO_{X \times \{y_0\}}) \subseteq 
\Supp(\E \lotimes \OO_{X \times \{y_0\}}) = \Supp(\E)\cap (X \times \{y_0\})\subseteq s_X^{-1}(\psi(s_Y(y_0))),
\]
where the first inclusion comes from the spectral sequence in \cite[(3.9)]{huybrechts}. Now 
$c_X'(\psi(s_Y(y_0))) = e_X$, because of $c_Y(y_0) = e_Y$ and the commutativity of \eqref{xy2.0}.}  
By using this, together  with  some calculations, it follows that  up to tensorizing with a trivial locally free sheaf of rank one, 
the line bundle  \eqref{eq:pformula} (on $\{(y_0,y_0)\}\times\Pic0 Z_Y$)  is isomorphic to ${\varphi_\E}_*(\cP_{Z}\otimes\OO_{\{e_X\}\times\Pic0 Z_X})$, which is trivial. This concludes the proof.
\end{proof}

The following variant will be useful in the next section. 
\begin{variant}\label{var:variant1}
Let $\bar \alpha\in \Pic0 X$ be a point parametrizing an $R$-stable line bundle $P_{\bar\alpha}$ \emph{(\S\ref{R1})}. 
Then for all $m\in \mathbb Z$ there
are isomorphisms
\[
(\Phi_{\E\, \boxtimes \,\ad\E}\boxtimes {\varphi_\E}_*)({\widetilde{\delta_X}}_*(q_X^*(\omega_X^{\otimes m}\otimes P_{\bar\alpha})\otimes \cP_{Z,X})) 
\; \cong \; {\widetilde{\delta_Y}}_*(q_Y^*(\omega_Y^{\otimes m}\otimes Q_{\varphi_\E(\bar\alpha)})\otimes \Q_{Z,Y}).
\]
\end{variant}
The proof is exactly the same as that of Theorem \ref{thm:invpoi}.

\subsection{Main result.}\label{Hdirect}

\begin{theorem}\label{thm:rel-direct-image} 
By keeping the previous notation and setting, there is an isomorphism of functors
\[
\mathbf R\,  {c_Y}_*\circ \mathbf L\delta_Y^* \; \cong \; \widehat{{\varphi_\E}}^* \circ \mathbf R\, {c_X}_*\circ \mathbf L\delta_X^*\circ \Phi_{\E\,\boxtimes\, \ad\E}\>,
\]
where $\widehat{{\varphi_\E}}$ denotes  the dual isomorphism of the Rouquier isomorphism 
and the equivalence $ \Phi_{\E\,\boxtimes\, \ad\E}$ is taken in the opposite direction:
\[
 \Phi_{\E\,\boxtimes\, \ad\E} \colon  D(Y\times Y)\rightarrow D(X\times X)\>.
 \]
\end{theorem}

\begin{proof}   To begin with, we claim that there are isomorphisms 
\begin{equation}\label{eq:fourier-mukai}
\Phi_{{\widetilde{\delta_X}}_*(\cP_{Z,X})} \; \cong \; \Phi_{\cP_Z}\circ \mathbf R {c_X}_*\circ \mathbf L\delta_X^*
\qquad\hbox{and}\qquad 
\Phi_{{\widetilde{\delta_Y}}_*(\Q_{Z,Y})}\; \cong \; \Phi_{\Q_Z}\circ \mathbf R {c_Y}_*\circ \mathbf L\delta_Y^*.
\end{equation}
In order to prove this, we first note that, by employing \cite[Proposition 2.1.2]{orlov} describing the kernel of the composition of integral functors,  the functors
\[
\Phi_{{\widetilde{\delta_X}}_*(q_X^*\omega_X^{\otimes m}\otimes \cP_{Z,X})} \colon D(X\times X)\rightarrow D(\Pic0 Z_X),\qquad\Phi_{{\widetilde{\delta_Y}}_*(q_Y^*\omega_Y^{\otimes m}\otimes \Q_{Z,Y})}: D(Y\times Y)\rightarrow D(\Pic0 Z_Y)
\]
verify the isomorphisms:
\begin{gather}\label{eq:rel-composition}
\Phi_{{\widetilde{\delta_X}}_*(q_X^*\omega_X^{\otimes m}\otimes \cP_{Z,X})}(-) \; 
\cong \; \Phi_{\cP_{Z,X}}(\mathbf L\delta_X^*(-)\otimes \omega_X^{\otimes m}) \\\notag 
\qquad \mbox{and} \qquad  \Phi_{{\widetilde{\delta_Y}}_*(q_Y^*\omega_Y^{\otimes m}\otimes \Q_{Z,Y})}(-) \; \cong\; \Phi_{\Q_{Z,Y}}(\mathbf L\delta_Y^*(-)\otimes \omega_Y^{\otimes m}).
\end{gather}
 For $m=0$ we get
 \begin{equation}\label{eq:rel-composition1}
\Phi_{{\widetilde{\delta_X}}_*( \cP_{Z,X})} \; \cong \; \Phi_{\cP_{Z,X}}\circ \mathbf L\delta_X^* \qquad \mbox{and} 
\qquad \Phi_{{\widetilde{\delta_Y}}_*( \mathcal{Q}_{Z,Y})} \; \cong \; \Phi_{\Q_{Z,Y}}\circ \mathbf L\delta_Y^*.
\end{equation}
On the other hand, since $\cP_{Z,X}=(c_X \times \mathrm{id})^*\cP_Z$ and $\Q_{Z,Y}=(c_Y \times \mathrm{id})^*\Q_Z$, from projection formula it follows  that
\begin{equation}\label{eq:composition2}
\Phi_{\cP_{Z,X}} \; \cong \; \Phi_{\cP_Z}\circ \mathbf R \, {c_X}_*\qquad \Phi_{\Q_{Z,Y}} \; \cong \; \Phi_{\Q_Z}\circ  \mathbf{R}\, {c_Y}_*\>.
\end{equation}
By plugging \eqref{eq:composition2} 
into \eqref{eq:rel-composition1} we get what claimed, namely \eqref{eq:fourier-mukai}. 

The next step consists in showing that Theorem \ref{thm:invpoi} yields the following isomorphism of functors, involving the Fourier-Mukai-Poincar\'e equivalences $\Phi_{\cP_Z} \colon D(\Alb Z_X)\rightarrow D(\Pic0 Z_X)$ and $\Phi_{\Q_Z} \colon D(\Alb Z_Y)\rightarrow D(\Pic0 Z_Y)$:
\begin{equation}\label{eq:rel-FM}
\Phi_{\Q_Z}\circ \mathbf R\,  {c_Y}_*\circ \mathbf L\delta_Y^* \; \cong \; {\varphi_\E}_* \circ \Phi_{\cP_Z}\circ \mathbf R\, {c_X}_*\circ \mathbf L\delta_X^*\circ \Phi_{\E\,\boxtimes\, \ad\E}\>,
\end{equation}
where  $\Phi_{\E\,\boxtimes\, \ad\E} $ is taken in the opposite direction.

In order to prove this fact,  we recall first   Orlov's relation  \cite[Proposition 2.1.6]{orlov} (or \cite[Exercise 5.13(ii)]{huybrechts}): given two equivalences $\Phi_{\F_i}:D(X_i)\rightarrow D(Y_i)$ for $i=1,2$, and an object $\mathcal R\in D(X_1\times X_2)$, then the image $\mathcal S:=\Phi_{\F_1\boxtimes\F_2}(\mathcal R)$ satisfies the relation
\[ 
\Phi_\mathcal S \; \cong \;  \Phi_{\F_2}\circ \Phi_\mathcal R\circ \Phi_{\F_1}\>,
\]
where the kernel $\F_1$ is used in the opposite direction, \emph{i.e.}   $\Phi_{\F_1}:D(Y_1)\rightarrow D(X_1)$. 
Now we apply this to 
$\Phi_{\E\, \boxtimes \,\ad\E}\boxtimes {\varphi_\E}_* \; \cong \; \Phi_{(\E\, \boxtimes \,\ad\E)\boxtimes\OO_{\Gamma_{\varphi_\E}}}$ and the object $\mathcal R={\widetilde{\delta_X}}_*(p_X^*\omega_X^{\otimes m}\otimes \cP_{Z,X})$. By combining
with Theorem \ref{thm:invpoi}, 
we get the relation
\[
\Phi_{{\widetilde{\delta_Y}}_*(q_Y^*\omega_Y^{\otimes m}\otimes \Q_{Z,Y})} \; \cong \;  {\varphi_\E}_* \circ \Phi_{{\widetilde{\delta_X}}_*(q_X^*\omega_X^{\otimes m}\otimes \cP_{Z,X})}\circ \Phi_{\E\,\boxtimes\, \ad\E}\>.
\]
For $m=0$ we get
 \begin{equation}\label{eq:rel-composition3}
\Phi_{{\widetilde{\delta_Y}}_*( \Q_{Z,Y})} \; \cong\; {\varphi_\E}_* \circ \Phi_{{\widetilde{\delta_X}}_*(\cP_{Z,X})}\circ \Phi_{\E\,\boxtimes\, \ad\E}\>.
\end{equation}
 By plugging \eqref{eq:fourier-mukai} into \eqref{eq:rel-composition3}, we  get \eqref{eq:rel-FM}. 

Finally, let us recall that a quasi-inverse  of the FMP equivalence $\Phi_{\mathcal Q_Z} \colon 
D(\Alb Z_Y)\rightarrow D(\Pic0 Z_Y)$  is $(-1_{\Alb Z_Y})^* \circ \Phi_{\mathcal Q_Z}[q]$, where now $\Phi_{\mathcal Q_Z}$ goes in the opposite direction and $q=\dim \Alb Z_Y$. We recall the commutativity  relation between isogenies of abelian varieties (in particular isomorphisms) and the FMP equivalence:
\[ 
\Phi_{\mathcal Q_Z}\circ {\varphi_\E}_* \; \cong\;  \widehat{\varphi_\E}^*\circ \Phi_{\mathcal P_Z}
\] 
(\cite[(3.4)]{mukai}). Therefore, by applying  the equivalence $(-1_{\Alb Z_Y})^*\circ \Phi_{\mathcal Q_Z}[q]$ to both sides of  \eqref{eq:rel-FM} we get the statement.
\end{proof}


\subsection{Relative  structures.}\label{HHoch}
We recall the Kostant-Hochschild-Rosenberg quasi-isomorphism (see e.g.\cite{calda}):
\begin{equation}\label{eq:HKR}
I\colon\>\mathbf L\delta_X^*({\delta_X}_*\omega_X^{\otimes m})\buildrel\sim\over\longrightarrow \bigoplus_{i=0}^{\dim X}L^{-i}\delta_X^*({\delta_X}_*\omega_X^{\otimes m})[i]=\bigoplus_{i=0}^{\dim X}\Omega^i_X\otimes\omega_X^{\otimes m}[i].
\end{equation}
By combining this with Theorem \ref{thm:rel-direct-image}, we get the derived invariance of the relative structure
\eqref{eq:relstr} of the Introduction.

\begin{corollary}\label{cor:hoch1}
 For any $m\in \mathbb Z$ we have isomorphisms
\[
\bigoplus_{q=0}^{\dim X}\widehat{\varphi_\E}^*\,\mathbf R\,  {c_X}_*(\Omega_X^q \otimes \omega_X^{\otimes m})[q]
 \; \cong \; \bigoplus_{q=0}^{\dim Y} \mathbf R\,  {c_Y}_*(\Omega_Y^q \otimes \omega_Y^{\otimes m})[q].
\]
\end{corollary}

\begin{proof} From \eqref{orlov-classical} it follows that 
\begin{equation}\label{eq:orlov-classical2}
\Phi_{\E\boxtimes \mathrm{ad}\E}({\delta_Y}_*\omega_Y^{\otimes m}) \; \cong \; {\delta_X}_*\omega_X^{\otimes m} \quad \quad \forall \, m\in \mathbb Z
\end{equation} 
(note: in \eqref{orlov-classical} this is stated ``in the other direction", \emph{i.e.}   for the functor $D(X\times X)\rightarrow D(Y\times Y)$ but  the statement of \eqref{orlov-classical} recovers    also \eqref{eq:orlov-classical2}). 
Hence by plugging $\delta_{Y*} \omega_Y^{\otimes m}$  into Theorem \ref{thm:rel-direct-image} we get isomorphisms
\[
\mathbf R\,  {c_Y}_*\circ \mathbf L\delta_Y^*({\delta_Y}_*\omega_Y^{\otimes m}) \; \cong \; \widehat{{\varphi_\E}}^* \circ \mathbf R\, {c_X}_*\circ \mathbf L\delta_X^*({\delta_X}_*\omega_X^{\otimes m}) \quad \quad \forall \, m\in \mathbb Z.
\]
Finally, by invoking the HKR isomorphisms \eqref{eq:HKR} we get the corollary.  
\end{proof}

By taking   $k$-th cohomology in the isomorphisms of the previous corollary, we obtain the following 
result.
\begin{corollary}\label{cor:homology} 
For any $m$ and $k$ there are isomorphisms
\[
 \bigoplus_{p-q=k}\widehat{\varphi_\E}^* \big( R^{p}{c_X}_*(\Omega_X^q\otimes \omega_X^{\otimes m}) \big) \; \cong \; 
 \bigoplus_{p-q=k}R^{p}{c_Y}_*(\Omega_Y^q\otimes \omega_Y^{\otimes m}).
\]
In particular, there are isomorphisms 
\[
\widehat{\varphi_\E}^*({c_X}_* \omega_X^{\otimes m}) \; \cong \; {c_Y}_* \omega_Y^{\otimes m} \quad \quad \forall \, m\in \mathbb Z.
\]
\end{corollary}
\begin{proof} The last statement is obtained for $k=-\dim X$.
\end{proof}

In analogy with Variant \ref{var:variant1}, we have the following variant of Corollary \ref{cor:homology}. 

\begin{variant}\label{var:variant2} 
Let $\bar\alpha\in \Pic0 X$ be a point parametrizing an $R$-stable line bundle on $X$ \emph{(see \S\ref{R1})}.  
For all $m$ and $k$ there are isomorphisms
\[
\bigoplus_{p-q=k}\widehat{\varphi_\E}^* \big( R^{p}{c_X}_*(\Omega_X^q\otimes \omega_X^{\otimes m}\otimes P_{\bar\alpha}) \big)
\; \cong \; \bigoplus_{p-q=k}R^{p}{c_Y}_*(\Omega_Y^q\otimes \omega_Y^{\otimes m}\otimes Q_{\varphi_\E(\bar\alpha)}).
\]
In particular, we obtain isomorphisms
\[
\widehat{\varphi_\E}^* \big(
{c_X}_*(\omega_X^{\otimes m}\otimes P_{\bar\alpha}) \big) \; \cong \; {c_Y}_*(\omega_Y^{\otimes m}\otimes Q_{\varphi_\E(\bar\alpha)}) 
\quad \quad \forall \, m\in \mathbb Z. 
\]
\end{variant}


\subsection{The relative canonical ring. } \label{Hcan}
Given  a morphism of projective varieties $f \colon X\rightarrow A$ we denote by
\[
\mathcal R(f) \; := \; \bigoplus_{m\ge 0}f_* \omega_X^{\otimes m}
\]
  the associated relative canonical algebra. 
We will show  the derived invariance of the multiplicative structure of 
$\mathcal R(c_X)$
 (in this paper  we will not address the problem of the 
 derived invariance of the multiplicative structure of $\mathcal{HA}(c_X)$).

\begin{theorem}\label{thm:invmult} 
With the notation as in the previous setting there is an isomorphism of $\OO_{\Alb Z_Y}$-algebras
\[
\widehat{\varphi_\E}^*\mathcal R(c_X) \; \cong \;  \mathcal R(c_Y).
\]
\end{theorem} 
\begin{proof} The isomorphism as graded $\OO_{\Alb Z_Y}$-modules follows from Corollary \ref{cor:homology}. What is left to prove is the multiplicativity of the isomorphism. This does not follow as easily as in the absolute case. Our argument will appeal to generic vanishing theorems of Hacon, Popa, Schnell and others, although we expect that there should be a self-contained argument. Specifically let 
\[
M_{k,n}(X) \colon ({c_X}_*\boxtimes {c_X}_*)( \omega_X^{\otimes k}\boxtimes \omega_X^{\otimes n})\rightarrow ({c_X}_*\boxtimes {c_X}_*)({\delta_X}_*\omega_X^{\otimes (k+n)}) \; \cong \; {\delta_{\Alb Z_X}}_* {c_X}_*\omega_X^{\otimes (k+n)}
\]
be the multiplication map (the map $\delta_{\Alb Z_X}$ is the diagonal embedding of $\Alb Z_X$). We need to prove that 
\begin{equation}\label{eq:multi}
(\widehat{\varphi_\E}^*\boxtimes\widehat{\varphi_\E}^*)(M_{k,n}(X)) \; \cong \; M_{k,n}(Y).
\end{equation}
 We apply the symmetric FMP transform $\FM_{\Alb Z_Y}$ (\S\ref{R4}) to \eqref{eq:multi}. Recall the usual commutation formulas between FMP transforms and morphisms of abelian varieties, say $f \colon A\rightarrow B$: 
\begin{equation}\label{eq:commutation}
\FM_B\circ \mathbf{R}f_* \; \cong \; {\bf L}\hat{f}^*\circ \FM_A,\qquad \FM_A\circ {\bf L}f^* \cong \mathbf{R}\hat{f}_*\circ \FM_B
\end{equation}
(\cite[Proposition 4.1]{schnell}). We have that
\[
\FM_{\Alb Z_Y}\circ \widehat{\varphi_\E}^*  \; \cong \; {\varphi_\E}_*\circ \FM_{\Alb Z_X}.
\]
Therefore \eqref{eq:multi} is equivalent to
\begin{equation}\label{eq:multi2}
{\varphi_\E}_*\bigl((\FM_{\Alb Z_X}\boxtimes \FM_{\Alb Z_X})(M_{k,n}(X)\bigr) \; \cong \; (\FM_{\Alb Z_Y}\boxtimes \FM_{\Alb Z_Y})(M_{k,n}(Y))\>.
\end{equation}
 Now  $(\FM_{\Alb Z_X}\boxtimes \FM_{\Alb Z_X})(M_{k,n}(X))$ coincides with  the following natural morphism in $D(\Pic0 Z_X\times\Pic0 Z_X)$:
\begin{equation}\label{eq:mor}
 (\FM_{\Alb Z_X}\boxtimes \FM_{\Alb Z_X})({\delta_{\Alb Z_X}}_* {c_X}_*\omega_X^{\otimes (k+n)} )\rightarrow \FM_{\Alb Z_X}({c_X}_*\omega_X^{\otimes k})\boxtimes\FM_{\Alb Z_X}({c_X}_*\omega_X^{\otimes n}) \>.
\end{equation}
By the generic vanishing theorems of Hacon and Popa-Schnell quoted in \S\ref{R3}, the complexes $\FM_{\Alb Z_X}({c_X}_*\omega_X^{\otimes h})$ are \emph{sheaves concentrated in degree $0$} for every  $h>0$, which we denote by  $\reallywidehat{({c_X}_*\omega_X^{\otimes h})^\vee}$. Therefore \eqref{eq:mor} is naturally identified to a morphism
\begin{equation}\label{eq:multi3}
m_{\Pic0 Z_X}^* \reallywidehat{({c_X}_*\omega_X^{\otimes (k+n)})^\vee} \longrightarrow \reallywidehat{({c_X}_*\omega_X^{\otimes k})^\vee}\boxtimes\reallywidehat{({c_X}_*\omega_X^{\otimes n})^\vee} \>,
\end{equation}
where $m_{\Pic0 Z_X} \colon \Pic0 Z_X\times \Pic0 Z_X\rightarrow \Pic0 Z_X$ is  the group law in $\Pic0 Z_X$, as by \eqref{eq:commutation} there is an isomorphism
\[ 
(\FM_{\Alb Z_X}\boxtimes \FM_{\Alb Z_X})\circ {\delta_{\Alb Z_X}}_* \; \cong \; m_{\Pic0 Z_X}^*\circ \FM_{\Alb Z_X}
\]
(the dual of the diagonal embedding of an abelian variety is the group law of the dual variety).  By base change and Serre duality,  
the fiber of the  map \eqref{eq:multi3} over any point $(\alpha, \alpha' )\in \Pic0 Z_X\times\Pic0 Z_X$  is a linear map
\[
H^0(\Alb Z_X, {c_X}_*\omega_X^{\otimes (k+n)}\otimes P_{\alpha+ \alpha' })^\vee\rightarrow H^0(\Alb Z_X, {c_X}_*\omega_X^{\otimes k}\otimes P_{\alpha})^\vee\otimes H^0(\Alb Z_X, {c_X}_*\omega_X^{\otimes n}\otimes P_{ \alpha' })^\vee
\]
\emph{i.e.}  
\[
H^0(X, \omega_X^{\otimes (k+n)}\otimes P_{\alpha+ \alpha' })^\vee\rightarrow H^0(X, \omega_X^{\otimes k}\otimes P_{\alpha})^\vee\otimes H^0(X, \omega_X^{\otimes n}\otimes P_{ \alpha' })^\vee .
\]
These maps are dual to the usual multiplication maps \eqref{eq:twisted}. The same holds for the fibers of the map $(\FM_{\Alb Z_Y}\boxtimes \FM_{\Alb Z_Y})(M_{k,n}(Y))$. 

Having said that, Theorem \ref{para:theorem} and its proof ensure that 
 \eqref{eq:multi2} holds \emph{at the fiber of every point $(\varphi_\E(\alpha),\varphi_\E( \alpha' ))\in \Pic0 Z_Y\times\Pic0 Z_Y$.}  By Nakayama's Lemma, this implies that  \eqref{eq:multi2} holds, since it is known that, for $k$ and $n$ positive, the scheme-theoretic support of both the source and  	 target of \eqref{eq:multi2} are the various components of the loci $-V^0( X, \omega_X^{\otimes h})$ (\emph{i.e.}   translates of abelian subvarieties), with their reduced scheme structure (see the Appendix  below).

In the case where either  $k$ or $n$ is zero the desired multiplicativity follows automatically because the map
\[ 
{c_X}_*\OO_X\otimes {c_X}_*\omega_X^{\otimes n}\rightarrow {c_X}_*\omega_X^{\otimes n}
\]
(\emph{i.e.}   the (${c_X}_*\OO_X$)-module structure of $ {c_X}_*\omega_X^{\otimes n}$) is induced canonically  from the identity $\mathrm{id}\in \mathrm{Hom}_{\OO_{\Alb Z_X}}({c_X}_*\omega_X^{\otimes n},{c_X}_*\omega_X^{\otimes n})$, via the following composition of  homomorphisms
\[
\xymatrix{ \mathrm{Hom}_{\OO_{\Alb Z_X}}({c_X}_*\omega_X^{\otimes n},{c_X}_*\omega_X^{\otimes n})\ar[r]^{\cong}&\mathrm{Hom}_{\OO_X}(c_X^*{c_X}_*\omega_X^{\otimes n},\omega_X^{\otimes n})\ar[d]\\
& \mathrm{Hom}_{\OO_{\Alb Z_X}}({c_X}_*c_X^*{c_X}_*\omega_X^{\otimes n}, {c_X}_*\omega_X^{\otimes n})\ar[d]^{\cong}\\&\mathrm{Hom}_{\OO_{\Alb Z_X}}({c_X}_*\OO_X\otimes {c_X}_*\omega_X^{\otimes n}\,,{c_X}_*\omega_X^{\otimes n})\>. }
\]
\end{proof}

\section{The relative canonical algebra of the Albanese morphism}\label{E0}

The aim of this section is to extend Theorem \ref{thm:invmult} to the Albanese morphism itself. The main difference with the  setting of  Theorem \ref{thm:invmult}  is that   Albanese varieties are  in general  not stable under the Rouquier isomorphism. A universally known   example of this
phenomenon is given by the Fourier-Mukai-Poincar\'e equivalence between an abelian variety $A$ and its dual $\Pic0 A$.  Its  Rouquier isomorphism exchanges $\Aut0 A$ (\emph{i.e.}   $A$ itself) with $\Pic0 A$. Therefore there are no non-trivial  $R$-stable subvarieties of $\Pic0 A$  in this case (see also \S\ref{ultima1} below). Therefore, in general, there is no hope to  dually recover from the Rouquier isomorphism any non trivial map between $\Alb Y$ and $\Alb X$, not even between quotients of them.

However it turns out that the relative canonical ring of the Albanese morphism is still preserved by the Rouquier isomorphism, although in a weaker form. This is the content of Theorem \ref{thmbasechange} below. Roughly speaking, 
we construct a finitely generated graded algebra  $\mathcal U_X$ of coherent sheaves on the Albanese-Iitaka variety
$\Alb Z_X$. Such algebra is preserved by the Rouquier isomorphism and its graded components carry  an essentially canonical direct-sum decomposition.  The relative canonical algebra $\mathcal R(a_X)$ 
is obtained from the algebra $p_X^*\mathcal U_X$ by twisting the pulbacks of the various summands by certain torsion line bundles in $\Pic0 X$. 


\subsection{Decomposition of direct images. }\label{R5} 
The main step towards the proof of 
 Theorem \ref{thmbasechange} is the following fact, already observed for $m=1$ by the first and third 
 author in \cite[Step 2 p.737]{capa}, and in \cite[Theorem D]{loposc} for  $m\geq 2$.  

\begin{proposition}\label{prop:simplified} 
Let $X$ be a smooth projective complex  variety and let 
$a_X\colon X\rightarrow \Alb X$ be its Albanese morphism. 
Moreover, let $c_X\colon X\rightarrow \Alb Z_X$ be the Albanese-Iitaka morphism and 
$p_X\colon\Alb X\rightarrow \Alb Z_X$  the natural quotient induced by  $c_X$.
\begin{itemize}
\item[(a)] For any $m \geq 1$ there is a  decomposition of the form
\begin{equation}\label{candec0}
{a_X}_* \omega_X^{\otimes m}  \; \cong  \;  \bigoplus_{j=1}^{N_X(m)}( {p_X}^* \G_{X,m, j} )\otimes P_{\alpha_{m,j}} \>,
\end{equation}
where:
\begin{itemize}
\item[(i)] the positive integer $N_X(m)$ is the one defined in \eqref{eq:non-van1} and \eqref{eq:non-van2};\\
\item[(ii)] the sheaves $\G_{X,m,j}$ are nonzero and $GV$  
 on $\Alb Z_X$  for all $m$ and $j$ \emph{(see \S \ref{R4})}; \\
\item[(iii)] the $\alpha_{m,j}$'s are torsion points in $\Pic0 X$ for all $m$ and $j$.  \\
\end{itemize}

\item[(b)] There are
isomorphisms
\begin{equation} {c_X}_*(\omega_X^{\otimes m}\otimes P_{\shortminus\alpha_{m,j}}) \; \cong \; \G_{X,m,j}\>
\end{equation}
for all $m$ and $j$.
\end{itemize}
\end{proposition}

\begin{proof} Point (a) follows from the existence of the \emph{Chen-Jiang decomposition} for direct images of pluricanonical sheaves under morphisms to abelian varieties. This is due to  J. Chen and Z. Jiang in \cite{cj} for canonical sheaves and generically finite morphisms, and for arbitrary morphisms to abelian varieties to \cite{paposc} (another proof is found in \cite{villadsen}).  It was extended to pluricanonical sheaves in \cite{loposc}. 
In the case of the Albanese morphism, the Chen-Jiang decomposition  is as follows:
\begin{equation}\label{eq:CJ}
{a_X}_*\omega_X^{\otimes m} \; \cong \; \bigoplus_i \rho_{m,i}^* \mathcal{F}_{m,i} \otimes P_{\gamma_{m,i}}, 
\end{equation}
where:
\begin{itemize}
\item[(i)]  each $\rho_{m,i} \colon \Alb X \rightarrow B_{m,i}$ is a quotient morphism of abelian varieties with connected fibers;\\
\item[(ii)] each $\mathcal{F}_{m,i}$ is a nonzero $M$-regular sheaf  on $B_{m,i}$ (see \S\ref{R4});\\
\item[(iii)] each $\gamma_{m,i} \in \Pic0 X$ is a torsion point;\\
\item[(iv)]  $\rho_{m,i}^*\Pic0 B_{m,i}-\gamma_{m,i} \ne   \rho_{m,j}^*\Pic0 B_{m,j} -\gamma_{m,j}$\  for $i\ne j$.
\end{itemize}

These decompositions are essentially canonical (see \cite[Remark 3.5]{Pa2}), and they are discussed at length in the references above and also in  \cite{Pa2}. We also refer to \S9 below where it is shown that the Chen-Jiang decompositions essentially carry the same information of the torsion filtration of the FMP transform, which shows once again that they are essentially canonical. 
It follows that
\begin{equation}\label{v0dec}
V^0(\Alb X, {a_X}_*\omega_X^{\otimes m}) \; = \; 
\bigcup_i \rho_{m,i}^*V^0(B_{m, i}, \mathcal{F}_{m,i }\otimes P_{\gamma_{m,i}}) \; = \; 
\bigcup_i \big( {\rho}^*_{m,i} \big(\Pic0 B_{m,i} \big) \shortminus \gamma_{m, i} \big),
\end{equation}
where the last equality holds because the sheaves $\F_{m,i}$ are $M$-regular, hence \eqref{eq:non-van} holds for them. Note that, by condition (iv), the translates of abelian subvarieties appearing in the decomposition \eqref{v0dec} are different from each other  (but some of them can be strictly contained in others). By comparing \eqref{v0dec} with \eqref{eq:non-van1} in the case  $m=1$, and \eqref{eq:non-van2} otherwise, it follows that  $\rho_{m,i}^*\Pic0 B_{m,i}$ is contained in $p_X^*\Pic0 Z_X$ for all $m$ and $i$, and therefore all quotient morphisms $\rho_{m,i}$ factor through $p_X \colon \Alb X\rightarrow \Alb Z_X$. 

Next, we observe that, still comparing \eqref{v0dec} with \eqref{eq:non-van1} in the case $m=1,$ and  
with \eqref{eq:non-van2} in the case $m>1$,
 we can gather those sheaves  $\rho_{m,i}^* \mathcal{F}_{m,i} \otimes P_{\gamma_{m,i}}$ appearing in the Chen-Jiang decomposition 
 of ${a_X}_*\omega_X^{\otimes m}$ such that the corresponding loci $V^0( \Alb X ,\rho_{m, i}^* \mathcal{F}_{m,i} \otimes P_{\gamma_{m,i}})$ are contained in \emph{the same translate of} $p_X^* \Pic0 Z_X$. This produces the decomposition \eqref{candec0}.

 \noindent (b) Since $c_X=p_X\circ a_X$, by applying ${p_X}_*$ to the right hand side of \eqref{candec0} twisted by $P_{\shortminus\alpha_{m,j}}$, the projection formula yields 
 \[
 {c_X}_*(\omega_X^{\otimes m}\otimes P_{\shortminus\alpha_{m,j}}) \; \cong \; \Bigl( \bigoplus_{i\ne j}\G_{X, m,j}\otimes p_{X*}(P_{\alpha_{m,i}\shortminus\alpha_{m,j}})\Bigr)\oplus \G_{X,m,j}.
 \]
 The summands in the right-hand side vanish except for the last one, because 
 $H^0 \big( (P_{\alpha_{m,i}\shortminus\alpha_{m,j}})|_{\ker (p_X) } \big) = 0$ for $i\ne j$. This is because,  by construction, for $i\ne j$ the line bundle  $P_{\alpha_{m,i} \shortminus \alpha_{m,j}}$ does not 
 belong to $\Pic0 Z_X=\ker\bigl(\Pic0(\Alb X)\rightarrow \Pic0(\ker (p_X))\bigr)$ (recall from \S\ref{R3} that $\ker p_X$ is connected). 
  \end{proof}
  
  Note that, as mentioned in Subsection \ref{R3}, the decompositions (\ref{candec0}) are the sheaf-theoretic versions of the decompositions of non-vanishing loci (\ref{eq:non-van1}) and (\ref{eq:non-van2}). In fact
  \begin{equation}\label{eq:non-van3}
  V^0(\Alb X,{a_X}_*\omega_X) = \bigsqcup_{i=1}^{N_X(1)}V^0(\Alb X, ( {p_X}^* \G_{X,1, j} )\otimes P_{\alpha_{1,j}})\subseteq \bigsqcup_{i=1}^{N_X(1)} p_X^*\Pic0 Z_X-\alpha_{1,i}\>
  \end{equation}
  and, for $m\ge 2$,
  \begin{equation}\label{eq:non-van4}
  V^0(\Alb X,{a_X}_*\omega_X^{\otimes m})  = \bigsqcup_{i=1}^{N_X(m)} V^0(\Alb X,( {p_X}^* \G_{X,m, j} )\otimes P_{\alpha_{m,j}}) = \bigsqcup_{i=1}^{N_X(m)} p_X^*\Pic0 Z_X-\alpha_{m,i}\>.
  \end{equation}
  
  \begin{remark}\label{rem:uniqueness} The decompositions of Proposition \ref{prop:simplified} are unique in the following sense: for any $m>0$ the integer $N_X(m)$ (the number of translates of $p_X^*\Pic0 Z_X$ that contain at least one  component of 
  $V^0(\Alb X, {a_X}_*\omega_X^{\otimes m})$) is obviously unique. Since \eqref{eq:non-van1} and \eqref{eq:non-van2} 
  hold for the torsion points $\alpha_{m,j}$ of Proposition \ref{prop:simplified}, it follows that the $\alpha_{m,j}$'s are unique modulo torsion points in $p_X^*\Pic0 Z_X$. Consequently, by (b) of the proposition above, also the sheaves $\G_{X,m,j}$ are unique up to tensorization with torsion line bundles parametrized by $\Pic0 Z_X$. 
  \end{remark}
  
  \begin{remark}
For $m\geq 2$ and all $j$  actually
the sheaves $\G_{X , m ,j}$  satisfy a stronger vanishing condition  called   index theorem with index zero.
Namely we have that $H^p ( \Alb Z_X , \G_{X , m , j } \otimes P_{\alpha} ) = 0$ for all $p\geq 1$ and $\alpha \in \Pic0 Z_X$.
 \end{remark}
  
  \begin{remark}\label{rem:finite-gen} As mentioned in Subsection \ref{R3}, after (\ref{eq:non-van2}), the fact that $\Pic0 Z_X$ is a subgroup of finite index of the group $G_X$ defined there follows also from the general finite generation result  \cite[Theorem 1.2]{bchm}. 
 To see this, we first recall that $V^0(X,\omega_X^{\otimes m})=V^0(\Alb X,{a_X}_*\omega_X^{\otimes m})$.  Let us consider   a component \ $ p_X^*\Pic0 Z_X\,\text{-}\,\alpha_{m,i}$ \ of $V^0(\Alb X, a_{X*}\omega_X^{\otimes m})$ and a component \  $p_X^*\Pic0 Z_X\,\text{-}\,\alpha_{k,h}$ \ of  $V^0(\Alb X, a_{X*}\omega_X^{\otimes k})$. Via multiplication of global sections, their product must map to a component of the form $p_X^*\Pic0 Z_X\,\text{-}(\alpha_{m,i}+\alpha_{k,h})$. By (\ref{eq:non-van3}) and (\ref{eq:non-van4}) such components correspond to the sheaves   $({p_X}^* \G_{X,m, j} )\otimes P_{\alpha_{m,j}}$ and $({p_X}^* \G_{X,k, h} )\otimes P_{\alpha_{k,h}}$ appearing in the decompositions (\ref{candec0}). Therefore, in the multiplicative structure of the relative canonical algebra ${\mathcal R}(a_X)$, their tensor product  must map to   a sheaf $( {p_X}^* \G_{X,m+k, i} )\otimes P_{\alpha_{m+k,i}} $ in the  decomposition (\ref{candec0}) of ${a_X}_*\omega_X^{\otimes (m+k)}$. Hence we have that, modulo $\Pic0 Z_X$,
  \[
  \alpha_{m+k,i}=\alpha_{m,j}+\alpha_{k,h}\>.
  \]
   It follows that the torsion group  $G_X/\Pic0 Z_X$ is finitely generated, hence finite, as soon as there is a finite set $\mathcal X$ of sheaves $({p_X}^* \G_{X,m, j}) \otimes P_{\alpha_{m,j}}$ appearing in the decompositions (\ref{candec0}) such that, for any other sheaf $\mathcal H_{r,s}:=({p_X}^* \G_{X,r, s})\otimes P_{\alpha_{r,s}}$, there is a tensor product of sheaves in the set ${\mathcal X}$ mapping, via the multiplicative structure of ${\mathcal R}(a_X)$, to $\mathcal H_{r,s}$. But this follows from the decompositions (\ref{candec0}) and the finite generation of ${\mathcal R}(a_X)$. 
    \end{remark}

\subsection{Main theorem.}\label{basechangesec0}
\begin{setting/notation}\label{notation}
The precise statement of Theorem \ref{thmbasechange}  will involve the decomposition of Proposition \ref{prop:simplified}. Keeping the setting of the previous Remark, we choose a section $s$ of the quotient of abelian groups $G_X\rightarrow G_X/\Pic0 Z_X$, and we denote $H_X$ the  image of $s$.
This choice determines  consistent  representatives for the points $\alpha_{m,j}\in\Pic0 X$ appearing in the decompositions (\ref{candec0}) and therefore consistent representatives for the sheaves $\G_{X,m,j}$
 in such a way that the sheaf $\G_{X, m,j}\otimes \G_{X, k, h}$ naturally maps to a sheaf $\G_{X, m+k, i}$ appearing in the decomposition of ${a_X}_*\omega_X^{\otimes (m+k)}$. This defines  a finitely generated graded algebra 
\begin{equation}\label{eq:algebra}
\mathcal U_X = \bigoplus_{m\geq 0} \mathcal{U}_{X,m}
\end{equation}
where $\mathcal{U}_{X,0}:={c_X}_*\OO_X$ and
\[ \mathcal{U}_{X,m} \; := \; \bigoplus_{j=1}^{N_X(m)}\G_{X,m,j} \quad \mbox{ for }\quad  m\geq 1\>.
\]
\end{setting/notation} 

\begin{theorem} \label{thmbasechange}  
 Let  $\Phi_\E\colon D(X)\rightarrow D(Y)$ be an equivalence and let us consider  decompositions as in Proposition \ref{prop:simplified}: 
\[
{a_X}_*\omega_X^{\otimes m}\cong \bigoplus_{j=1}^{N_X(m)}( {p_X}^* \G_{X,m, j} )\otimes P_{\alpha_{m,j}} \qquad \mbox{and} \qquad {a_Y}_*\omega_Y^{\otimes m}\cong 
\bigoplus_{j=1}^{N_Y(m)}( {p_Y}^* \G_{Y,m, j} )\otimes Q_{\beta_{m,j}}. 
\]
Then: 
\begin{itemize}
\item[(1)] Given a subgroup $H_X$ of $ \Pic0 X$  as in Setting and Notation \ref{notation}, the subgroup $H_Y:=\varphi_\E(H_X)$ of $\Pic0 Y$ is the image of a section of $G_Y/\Pic0 Z_Y$, determining consistent representatives $\G_{Y,m,j}$ such that, up to reordering, for all $m\geq 1$ and $j$ we have
\[
\widehat{\varphi_\E}^*\G_{X,m,j} \;  \cong \; \G_{Y,m,j}
\]
\emph{(here, as usual,  $\widehat{\varphi_\E}\colon \Alb Z_Y\rightarrow \Alb  Z_X$ is the dual  isomorphism of the induced Rouquier isomorphism \eqref{eq:PicZ}, and $\varphi_\E\colon G_X\rightarrow G_Y$ is the   isomorphism \eqref{eq:G}).}
 In particular, $N_X(m)=N_Y(m)$.\\

\item[(2)] There is an isomorphism
\[
\widehat{\varphi_\E}^*\mathcal{U}_{X} \; \cong \; \mathcal \U_Y
\]
as graded algebras.
\end{itemize}
\end{theorem}

\begin{proof} (1)  By construction $H_X$  is a subgroup of the group $G_X$ 
hence it is $R$-stable. Therefore point  $(1)$ follows from  Variant \ref{var:variant2} and Proposition \ref{prop:simplified}(b). This provides an isomorphism  of graded $\OO_{\Alb Z_Y}$-modules $\widehat{\varphi_\E}^*\mathcal{U}_{X}\cong \mathcal \U_Y$. Note that the last assertion was already proved in \eqref{eq:N(m)}. 

\noindent To prove $(2)$ it remains to prove the invariance of the multiplicative structure. 
This is done exactly as in the proof of Theorem \ref{thm:invmult} by considering, rather than the multiplication maps $M_{k,n}(X)$, the multiplication
\[
({c_X}_*\boxtimes {c_X}_*) \big( (\omega_X^{\otimes k}\otimes P_{- \alpha_{k,j}})\boxtimes (\omega_X^{\otimes n}\otimes P_{ - \alpha_{n,h}}) \big)\rightarrow ({c_X}_*\boxtimes {c_X}_*) \big( {\delta_X}_*(\omega_X^{\otimes (k+n)}\otimes P_{ - \alpha_{k,j} - \alpha_{n,h}}) \big)
\]
and the correspondent multiplication of coherent sheaves on $\Alb Z_Y$.  (To this purpose we recall that the sheaves  ${c_X}_*(\omega_X^{\otimes m}\otimes P_{  \alpha_{m,j}}) $ are $GV$ as well for $m\geq 1$ and $j$, 
as it can be  seen by reducing to untwisted pluricanonical bundles on  the \'etale cover of $X$ induced by the torsion line bundle $P_{- \alpha_{m,j}}$). 
\end{proof}

It would be interesting to know whether a similar invariance holds true for relative structures, as in Corollary \ref{cor:hoch1}. 

\subsection{Derived invariance of the Chen-Jiang decomposition. } 
In answer to a question of Mihnea Popa, we show that not only 
the  direct images of pluricanonical sheaves under the Albanese map are derived invariant, 
 but also their Chen-Jiang decompositions are so. 

In the first place we claim that the sheaves   $\G_{X,m,j}$ have a Chen-Jiang decomposition (see the Appendix). This  because, for a torsion point $\alpha\in \Pic0 X$, the sheaf $c_{X*}(\omega_X^{\otimes m}\otimes P_\alpha)$  is a direct summand of the pushforward of the $m$-th pluricanonical sheaf of  the cyclic \'etale cover of $X$, say $X^\prime$ under the morphism $X^\prime \rightarrow \Alb Z_X$, and as such it has its own Chen-Jiang decomposition (this follows from \cite[Proposition 3.6]{loposc}). 
Therefore the claim follows from Proposition \ref{prop:simplified}(b). 

\begin{lemma}\label{prop:pop1} 
For all $m>0$ the isomorphisms $\widehat{\varphi_\E}^*\G_{X,m,j} \cong \G_{Y,m,j}$ of Theorem \ref{thmbasechange}(1) are such that $\widehat{\varphi_\E}^*$ transforms  a Chen-Jiang decomposition of  $\G_{X,m,j} $ to a Chen-Jiang decomposition of $\G_{Y,m,j}$. 
\end{lemma}
\begin{proof} We claim that, by applying the symmetric FMP transform to the isomorphisms of the statement, there are
 isomorphisms of sheaves 
\[
\varphi_{\E*}\reallywidehat{\G_{X,m,j}^\vee} \; \cong \; \reallywidehat{\G_{Y,m,j}^\vee}.
\] 
Indeed we know that such isomorphisms hold at the level of the usual FMP transforms (\emph{i.e.}   \eqref{eq:rel-FM}). 
 Hence  they also hold  for the 
 symmetric FMP transform: just apply the dualizing functor to both sides of \eqref{eq:rel-FM} and use the fact that, by Grothendieck duality, $\Delta_{\Pic0 Z_Y}\circ \Phi_{\Q_Y}\cong ((-1)^*\circ \Phi_{\Q_Y}\circ \Delta_{\Alb Z_Y})[ \dim \Pic0 Z_Y]$ (\cite[(3.8)]{mukai}; the details are left to the reader). 
Obviously such isomorphisms preserve the torsion filtrations  of \S \ref{rem:support}. Therefore, by the conclusion of  \S \ref{rem:support},
the isomorphisms of the statement preserve the Chen-Jiang decompositions.
\end{proof}

Next we observe that by its very definition, given a  Chen-Jiang decomposition
\[
\G_{X,m,j} \; \cong \; \bigoplus_i \rho_{m,j,i}^* \mathcal{F}_{m,j,i} \otimes P_{\gamma_{m,j,i}}, 
\]
then
\[
p_X^*\G_{X,m,j} \; \cong \;  \bigoplus_i (\rho_{m,j,i}\circ p_X)^* \mathcal{F}_{m,j,i} \otimes p_X^*P_{\gamma_{m,j,i}}, 
\]
is a Chen-Jiang decomposition. Moreover, from the definition of the decompositions of Proposition \ref{prop:simplified}, it follows that 
\begin{equation}\label{eq:new-CJ} {a_X}_*\omega_X^{\otimes m} \; \cong \; \bigoplus_{i,j} (\rho_{m,j,i}\circ p_X)^* \mathcal{F}_{m,j,i} \otimes p_X^*P_{\gamma_{m,j,i}} \otimes P_{\alpha_{m,j}} 
\end{equation}
is a Chen-Jiang decomposition. 

Let us place ourselves  in the setting of Theorem \ref{thmbasechange}(1). Recalling the essential uniqueness of Chen-Jiang decompositions, it follows from Lemma \ref{prop:pop1}  and (\ref{eq:new-CJ}) that Chen-Jiang decompositions of  
pushforwards of pluricanonical sheaves under the Albanese map  are derived invariants:
\begin{proposition}\label{prop:pop} For all $m \geq 1$
\[
{a_Y}_*\omega_Y^{\otimes m} \; \cong \; \bigoplus_{i,j}    \bigl( \rho_{m,j,i}\circ \widehat{\varphi_\E} \circ p_Y \bigr)^* \mathcal{F}_{m,j,i} \otimes p_Y^*Q_{\varphi_{\E} ( \gamma_{m,j,i}  ) }  \otimes Q_{\beta_{m,j}} 
\]
is a Chen-Jiang decomposition of  ${a_Y}_*\omega_Y^{\otimes m}$.
\end{proposition}
Notice that, in particular,  it follows  from the above Proposition that, for all $m>0$, the  quotient 
abelian varieties   of $\Alb X$ and $\Alb Y$  involved in the 
Chen-Jiang decompositions of ${a_X}_*\omega_X^{\otimes m}$ and  ${a_Y}_*\omega_Y^{\otimes m}$, respectively,  are isomorphic.

\subsection{Derived invariance of the canonical ring of the general fiber. } 
Starting from the present subsection, in the rest of the paper we will draw some applications of Theorem \ref{thmbasechange}. We begin with the derived 
 invariance of the canonical ring of the general fiber of the Albanese map.

Given a quasi-coherent sheaf $\F$ on a complex variety $Z$ and a closed point $z\in Z$, 
we denote by  $\F(z):=\F\otimes \mathbb C(z)$ the fiber of $\F$ at $z$ equipped with  
its structure of $\mathbb C$-vector space. Given a smooth projective variety $X$, we denote, as above, 
by $\mathcal R(a_X)$  the relative canonical algebra under the Albanese morphism. For $x\in X$, 
we let $z=a_X(x)$. We consider the fiber of $\mathcal R(a_X)$ at $z$:
\[
\mathcal R(a_X)(z)  \cong (\bigoplus_{m\ge 0}{a_X}_*\omega^{\otimes m}_X)(z)\> .
\]
 We denote by  $F_{X,z}$ the fiber of the Albanese map at $z$ (note that fibers may be disconnected, 
 in which case a general fiber is the disjoint union of general fibers of the Stein factorization). Let 
 \[
 R(F_{X,z})  \; = \;  \bigoplus_{m\ge 0}H^0(F_{X,z,}\, \omega_{F_{X,z}}^{\otimes m})
 \]
be the canonical ring of $F_{X,z}$, and let 
  \[
  (F_{X,z})_{can} \; = \; \Proj(R(F_{X,z}))\>
 \]
be the canonical model.

\begin{corollary}\label{cor:koddim} In the setting of  Theorem \ref{thmbasechange}, let $x\in X$ and $y\in Y$ be  such that $\widehat{\varphi_\E}(c_X(x))=c_Y(y)$. Let also $z=a_X(x)$ and  $t=a_Y(y)$. 
Then,
 for general $x\in X$,  one has that
   $$ R(F_{X,z}) \; \cong \; R(F_{Y,t}) $$  
as $\mathbb C$-algebras. 
In particular, $(F_{X, z})_{can}\cong (G_{Y, t})_{can}$. Hence the Kodaira dimension of the general fiber 
of the Albanese map is invariant under derived equivalence.
\end{corollary}

\begin{proof} It follows from Theorem \ref{thmbasechange} that, 
for $x\in X$ and $y\in Y$ as in the hypothesis, we have $\mathbb C$-algebra isomorphisms
\begin{eqnarray*}
\bigoplus_{m\ge 0}({a_X}_*\omega_X^{\otimes m})(z)&\cong &\bigoplus_{m\ge 0} \bigl(\bigoplus_{j=1}^{N_X(m)}( {p_X}^* \G_{X,m, j} )(z)\otimes P_{\alpha_{m,j}}  (z)\bigr)\\
&\cong & \bigoplus_{m\ge 0}(p_X^*(\bigoplus_{j=1}^{N_X(m)}( \G_{X,m, j} ) )(z)\\
& \cong & \bigoplus_{m\ge 0}(p_Y^*(\bigoplus_{j=1}^{N_Y(m)}( \G_{Y,m, j} ) )(t)\\
&\cong & \bigoplus_{m\ge 0}\bigl(\bigoplus_{j=1}^{N_Y(m)}( {p_Y}^* \G_{Y,m, j} )(t)\otimes Q_{\beta_{m,j}}(t)\bigr)\\
&\cong &\bigoplus_{m\ge 0}({a_Y}_*\omega_Y^{\otimes m})(t)
\end{eqnarray*}
where the third isomorphism is given by (2) of Theorem \ref{thmbasechange}.

The asserted isomorphism follows from this because, by base-change, the fiber at a general point of the relative canonical ring is canonically isomorphic to the canonical ring of the fiber.
\end{proof}


\section{The relative canonical model of the Albanese map}\label{E1}

In this section we apply the results of the previous section to study the invariance under derived equivalence of
\[
X^{can}_a \; := \; \Proj \big( \mathcal R(a_X) \big)  \longrightarrow \Alb X,
\] 
the relative canonical model of the Albanese morphism.
In the spirit of the recent work \cite{itoetal} of Ito-Miura-Okawa-Ueda we show that the relative Albanese canonical models of two derived-equivalent varieties have the same class in the Grothendieck ring modulo isogenies of abelian varieties. In the final subsection we draw an application to the derived invariance of the Hodge numbers $h^{0,j}$ for varieties relatively of general type with respect to the Albanese morphism.

 \subsection{Main result.}\label{Can1} 
We will consider the
 Grothendieck ring $K_0(\mathrm{Var}/{\mathbb{C}})$ of algebraic varieties over $\mathbb{C}$. 
 Actually, for our purposes,  it is  useful to take a certain quotient of $K_0(\mathrm{Var}/{\mathbb{C}})$ introduced in \cite{itoetal}: 
\begin{equation}\label{quoimou}
K'_0(\mathrm{Var}/{\mathbb{C}})
\end{equation}
 is the quotient of $K_0({\rm Var}/{\mathbb{C}})$ by the ideal generated by $\big[A\big] - \big[B\big]$, where $A$ and $B$ are isogenous abelian varieties.
Motivated by a result of Orlov asserting that derived equivalent abelian varieties are isogenous,
 the
 authors  of \cite{itoetal}
ask whether derived equivalent varieties have the same class in the localized Grothendieck ring $K'_0(\mathrm{Var}/{\mathbb{C}})[\mathbb{L}^{-1}]$, with $\mathbb{L}:=\big[\mathbb{A}^1\big]$ be the class of the affine line (\emph{op.cit.}, Problem 7.2).
We prove that in any case -- even without localizing --  the relative Albanese canonical models of derived equivalent varieties do: 
 		\begin{theorem}\label{thmGroclass}
Let $X$ and $Y$ be derived equivalent varieties.
	Then 	the equality
\begin{equation}\label{eqprime}
\big[X^{can}_a\big] = \big[Y^{can}_a\big] 
\end{equation}
holds in $K'_0(\mathrm{Var}/{\mathbb{C}})$.
		\end{theorem}

\begin{proof}
We go back to the setting of \S \ref{E0}.
Let $\mathcal{U}_{X}$ be the algebra of sheaves over $\Alb Z_X$ introduced in \S\ref{basechangesec0}. 
It is a derived invariant by Theorem \ref{thmbasechange}(2). 
In particular, the structure map $\Proj\bigl(\mathcal{U}_{X}\bigr)\rightarrow \Alb Z_X$ is a derived invariant.
First, we claim that the equality
\begin{equation}\label{relcan1}
\big[\Proj \bigl(\mathcal R(a_X) \bigr)] = \big[\Proj\bigl(p_X^*\mathcal U_X \bigr)\big]
\end{equation} 
holds in the  Grothendieck ring of varieties $K_0(\mathrm{Var}/{\mathbb{C}})$. 
To prove this, recall that the two algebras $\mathcal R(a_X)$ and $p_X^*\mathcal U_X$
only differ for the twisting by the torsion line bundles appearing in the finite semigroup $\{P_{\alpha_{m, j}}\}_{m \geq 1, j}$ (Theorem \ref{thmbasechange}). 
By choosing a finite open covering of the quasi-compact variety $\Alb X$ given by a common local trivialization of the  $P_{\alpha_{m, j}}$'s,
we see that there exist finite decompositions
\[
\Proj \big( \mathcal R(a_X)) = \bigsqcup_{i=1}^s V_i \quad \quad \textrm{and} \quad \quad 
\Proj \big(p_X^*\mathcal U_X\big)  = \bigsqcup_{i=1}^s W_i,
\]
 where $V_1, \ldots , V_s, W_1, \ldots , W_s$ are locally closed subvarieties, and $V_i$ is isomorphic to $W_i$ for all $i= 1, \ldots , s$. 
Namely,  $\Proj \bigl( \mathcal R(a_X))$ and  $\Proj\bigl(p_X^*\mathcal U_X)$  are \emph{piecewise isomorphic}. Therefore, 
\eqref{relcan1}
holds in $K_0(\mathrm{Var}/\mathbb{C})$ by definition. 

Note that the same happens for $Y$. Now, thanks to the functoriality of the  $\Proj$ (\cite[Prop.\ 3.5.3]{egaII}), we have that
\[\Proj\bigl(p_X^* \U_{X} \bigr)\simeq \Proj\bigl( \U_{X}\bigr)\times_{\Alb Z_X} \Alb X,
\]
and similarly for $Y$.
 This is expressed by the cartesian  diagrams
\begin{equation}\label{cartHodge}
\xymatrix{
\Proj \big( p_X^* \U_{X} \big)    \ar[r] \ar[d] & \Alb X \ar[d]^{p_X}  \\
\Proj \big(  \U_{X} \big)    \ar[r] & \Alb Z_X }
\quad \quad \quad \quad \xymatrix{
\Proj \big( p_Y^*  \U_{Y} \big)  \ar[r] \ar[d] & \Alb Y   \ar[d]^{p_Y}  \\
\Proj \big(  \U_{Y} \big)    \ar[r] & \Alb Z_Y  }
\end{equation}
Since $p_X$ is a surjective morphism of abelian varieties with connected fibers, it follows from a standard property on the behaviour of Grothendieck rings (see, e.g., \cite[Propositions 2.3.3 and 2.3.4]{motivicint}) that 
\begin{equation}\label{gro2}
\big[\Proj \big( p_X^* \U_{X} \big) \big] =\big[\ker (p_{X})\big]\big[\Proj \big(  \cU_{X} \big)\big] 
\end{equation}
in $K_0(\mathrm{Var}/\mathbb{C})$,
and the same  holds for $Y$. 

\begin{claim}\label{claim} The abelian varieties $\ker(p_{X})$ and $\ker(p_{Y})$ are isogenous.
\end{claim} The Theorem follows from this since $\Proj \big( \cU_{X} \big)$ and $\Proj \big( \cU_{Y} \big)$ are isomorphic by Theorem  \ref{thmbasechange}(2). 

\begin{proof}[Proof of  Claim \ref{claim}]
By a result of Popa and Schnell (\cite[Theorem A (1)]{ps1}) the abelian varieties $\Alb X$ and $\Alb Y$ are isogenous. Therefore, by Poincar\'e's complete reducibility $\ker p_X\times\Alb Z_X$ is isogenous to $\ker p_Y\times\Alb Z_Y$, and we know that $\Alb Z_X$ and $\Alb Z_Y$ are isomorphic.  This implies that $\ker p_X$ and $\ker p_Y$ are isogenous by the cancellation theorem of Fujita (\cite[Proposition 9]{fujita}). 
\end{proof}
\end{proof}

\subsection{Derived invariance of Hodge numbers $h^{0, j}$}\label{final}

	As observed in \cite{itoetal}, the usefulness of the quotient ring $K'_0(\mathrm{Var}/\mathbb{C})$ is that the Hodge-Deligne polynomial factors through it. Namely,  
let 
\[
{\rm HD} \colon K_0(\mathrm{Var}/{\mathbb{C}}) \rightarrow \mathbb{Z}[u, v]
\]
be the ring homomorphism that associates to each class in $K_0(\mathrm{Var}/{\mathbb{C}})$ its Hodge-Deligne polynomial.  
For the class of a smooth and projective variety $Z$, one has 
\begin{equation}\label{smHD}
{\rm HD} (\big[Z\big]) = \sum_{i, j \geq 0} (-1)^{i+j} h^{i, j}(Z)u^iv^j. 
\end{equation}
Since isogenous abelian varieties have the same Hodge numbers,  $\rm{HD}$ passes to the quotient and we obtain an induced ring homomorphism
\[
{\rm HD}' \colon K'_0(\mathrm{Var}/{\mathbb{C}}) \rightarrow \mathbb{Z}[u, v]
\]
 such that \eqref{smHD} continues to hold. 
	\begin{theorem}\label{invHN1}
Let $X$ and $Y$ be derived equivalent varieties. Assume that the generic fiber of the Albanese morphism of $X$ is of general type, that is $\omega_X$ is $a_X$-big.  Then we have
\[
h^{0, j}(X) \; = \;  h^{0, j}(Y),
\]
for all $j$.
\end{theorem}
\begin{proof}
We consider again diagrams \eqref{cartHodge}, keeping in mind that the two bottom rows are isomorphic by  Theorem \ref{thmbasechange}(2). Let $W$ be   a resolution of singularities of  $\Proj \big( \U_{X} \big) $, and let $T_X\rightarrow W$ and $T_Y\rightarrow W$ be the corresponding base extensions of $ \Proj \big( p_X^* \U_{X} \big)  $ and $\Proj \big( p_Y^*  \U_{Y} \big)$, respectively. We have that $T_X$ and $T_Y$ are smooth and, repeating the argument of Theorem \ref{thmGroclass} (as in \eqref{gro2}), the equality $[T_X]=[T_Y]$ holds true
in $K'_0(\mathrm{Var}/{\mathbb{C}})$. Hence
\[
{\rm HD}'([T_X])\; = \;  {\rm HD}'([T_Y]),
\]
so that  all Hodge numbers of $T_X$ and $T_Y$ are equal.  
As already noted in the proof of  Theorem \ref{thmGroclass}, 
there are isomorphic non-empty open subsets of $\Proj \big( p_X^* \mathcal{U}_{X} \big)$ and $\Proj \big( \mathcal R(a_X) \big)$.
Since $\Proj \big(p_X^* \mathcal{U}_{X} \big)$ and $\Proj \big( \mathcal R(a_X) \big)=X^{can}_a$ are irreducible,\footnote{
Thanks to \cite[Cor.\ 2.3.5 $(iii)$]{egaiv} and \eqref{cartHodge}, $\Proj \big( p_X^* \mathcal{U}_{X} \big)$ is irreducible because $\Proj \big( \mathcal{U}_{X} \big)$ is so by \cite[Prop.\ 3.1.14 $(i)$]{egaII}. The irreducibility of $\Proj \big( \mathcal R(a_X) \big)$ is clear.} they are indeed birational.\footnote{In the next \S \ref{lastsection0} we will prove that $X^{can}_a$ and $\Proj \big(p_X^* \mathcal{U}_{X} \big)$ are  isomorphic under the assumption that $\omega_X$ is $a_X$-big (see Lemma \ref{lemiso} below). However, at this stage, this result is not necessary.}
 Now the hypothesis that the generic fiber of $a_X$ (and hence of $a_Y$, by Corollary \ref{cor:koddim}) is of general type yields that $X$ is birational to $X^{can}_a$ and  in turn $Y$ is birational to $Y^{can}_a$. Hence $X$ (\emph{resp}. $Y$) is birational to $T_X$ (\emph{resp}. to $T_Y$). Hence $h^{0,j}(X)=h^{0,j}(T_X)=h^{0,j}(T_Y)=h^{0,j}(Y)$. 
\end{proof}


\section{Derived invariance of Hodge numbers for minimal varieties}\label{lastsection0}

In this section we prove the invariance of the Hodge numbers of smooth minimal varieties relatively of general type with respect to their Albanese morphism.
\begin{theorem}\label{invHminimal}
Let $X$ and $Y$ be derived equivalent varieties. Assume 
that  $\omega_X$ is nef and $a_X$-big.\footnote{Actually, since an abelian variety contains no rational curves, the 
nefness of $\omega_X$  is equivalent to its $a_X$-nefness (see, e.g.,  \cite[\S 3.2]{birkarchen}).} Then the rational Hodge structures of $X$ and $Y$ are isomorphic. In particular,
we have 
\[
h^{i, j}(X) = h^{i, j}(Y)
\]
for all $i$ and $j$.
\end{theorem}
This will be the consequence of a result involving a variant of the Grothendieck ring of varieties, for which we recall and introduce some  terminology. Let
\[
\mathcal{M}_{\mathbb{C}}^{\prime} := K_0^{\prime}(\mathrm{Var} / \mathbb C)[\mathbb{L}^{-1}]
\] 
be the localization of the Grothendieck ring $K_0^{\prime}(\mathrm{Var} / \mathbb C)$  by the class of the affine line $\mathbb{L}$ introduced in \cite{itoetal} (see \S \ref{Can1}).
Then we have the completion 
\[
\widehat{\mathcal{M}_{\mathbb{C}}^{\prime}} := \varprojlim_{d \in \mathbb{Z}}\ \frac{\mathcal{M}_{\mathbb{C}}^{\prime}}{F_d \mathcal{M}_{\mathbb{C}}^{\prime}}
\]
of $\mathcal{M}_{\mathbb{C}}^{\prime}$ with respect to the filtration induced by dimension
$(F_{\bullet} \mathcal{M}_{\mathbb{C}}^{\prime})$. This  is defined as follows. For every integer $d$, $F_d \mathcal{M}_{\mathbb{C}}^{\prime}$ is the subgroup of $\mathcal{M}_{\mathbb{C}}^{\prime}$ generated by the elements of the form $[Z] \mathbb{L}^{-m}$, where $Z$ is an algebraic variety such that  
$\dim Z - m \leq d$. This is a
ring filtration of $\mathcal{M}_{\mathbb{C}}^{\prime}$, in the sense that  $F_d \mathcal{M}_{\mathbb{C}}^{\prime} \cdot F_k \mathcal{M}_{\mathbb{C}}^{\prime} \subseteq F_{d+ k} \mathcal{M}_{\mathbb{C}}^{\prime}$, so $\widehat{\mathcal{M}_{\mathbb{C}}^{\prime}}$ is also a ring.

The ring $\widehat{\mathcal{M}_{\mathbb{C}}^{\prime}}$ was defined analogously to the completed Grothendieck ring of varieties  $\widehat{\mathcal{M}_{\mathbb{C}}}$ (the completion of the localized Grothendieck ring $\mathcal{M}_{\mathbb{C}} := K_0(\mathrm{Var} / \mathbb C)[\mathbb{L}^{-1}]$ with respect to the same filtration by dimension). There is the natural  ring homomorphism
\begin{equation}\label{ringhom1}
\widehat{\mathcal{M}_{\mathbb{C}}} \rightarrow \widehat{\mathcal{M}_{\mathbb{C}}^{\prime}}.
\end{equation}

\begin{theorem}\label{invminimal} Under the assumptions of Theorem \ref{invHminimal} the
 equality $[X] = [Y]$
holds  in 
$\widehat{\mathcal{M}_{\mathbb{C}}^{\prime}}$.
\end{theorem}
\noindent The proof of Theorem \ref{invminimal} will be the content of the next subsections, starting from 7.1.

\begin{proof}[Proof of Theorem \ref{invHminimal}]  The result follows from Theorem \ref{invminimal} in the usual way (see, e.g., \cite[Chapter 2, Corollary 4.3.8]{motivicint}). We include the details for the reader's benefit. From the mixed Hodge theory of Deligne, it follows the existence of a ring homomorphism
\[
\chi_{\mathrm{Hdg}} \colon K_0(\mathrm{Var} / \mathbb C) \rightarrow K_0(\mathbf{pHS})
\] 
from $K_0(\mathrm{Var} / \mathbb C)$ to the Grothendieck ring of the \emph{semisimple} category $\mathbf{pHS}$ of polarizable Hodge structures \cite[Chapter 2, \S 3.2]{motivicint}. If a variety $X$ is smooth and proper, then 
\[
\chi_{\mathrm{Hdg}}([X]) = \sum_{k \geq 0} (-1)^k [H^k(X, \mathbb Q)]
\]
is the class of its rational Hodge structure in $K_0(\mathbf{pHS})$. Isogenous abelian varieties have isomorphic rational Hodge structures, hence $\chi_{\mathrm{Hdg}}$ passes to the quotient $K_0^{\prime}(\mathrm{Var} / \mathbb C)$
\[
\chi_{\mathrm{Hdg}}^{\prime} \colon K_0^{\prime}(\mathrm{Var} / \mathbb C) \rightarrow K_0(\mathbf{pHS}).
\]
Moreover, $\chi_{\mathrm{Hdg}}(\mathbb L)$ is invertible in $K_0(\mathbf{pHS})$ (indeed, it coincides with the inverse of the class of the Tate object $\mathbb{Q}(1)$ in $K_0(\mathbf{pHS})$). Therefore, we get an induced ring homomorphism
\begin{equation}\label{chiHdg}
\chi_{\mathrm{Hdg}}^{\prime} \colon \mathcal{M}_{\mathbb{C}}^{\prime} \rightarrow K_0(\mathbf{pHS}).
\end{equation}
Let $\overline{\mathcal{M}_{\mathbb{C}}^{\prime}}$ be the image of the map $\mathcal{M}_{\mathbb{C}}^{\prime} \to \widehat{\mathcal{M}_{\mathbb{C}}^{\prime}}$. By definition, it is isomorphic to $\mathcal{M}_{\mathbb{C}}^{\prime} / F_{\infty}\mathcal{M}_{\mathbb{C}}^{\prime}$, where $F_{\infty}\mathcal{M}_{\mathbb{C}}^{\prime} := \bigcap_{d \in \mathbb Z} F_{d}\mathcal{M}_{\mathbb{C}}^{\prime}$. We want to prove that \eqref{chiHdg} factorizes through  $\overline{\mathcal{M}_{\mathbb{C}}^{\prime}}$. In order  to achieve this, one needs   that $\chi_{\mathrm{Hdg}}^{\prime}(F_{\infty}\mathcal{M}_{\mathbb{C}}^{\prime}) = 0$. For every integer $d$, let $W_d K_0(\mathbf{pHS})$ be the subgroup of $K_0(\mathbf{pHS})$ generated by the polarizable Hodge structures of weight $\leq d$.
Since the subgroup $F_{d}\mathcal{M}_{\mathbb{C}}^{\prime}$ is generated by the elements of the form $[Z] \mathbb{L}^{-m}$ with $\dim Z - m \leq d$, we have 
\[
\chi_{\mathrm{Hdg}}^{\prime}([Z] \mathbb{L}^{-m}) = \chi_{\mathrm{Hdg}}^{\prime}([Z]) \chi_{\mathrm{Hdg}}^{\prime}(\mathbb{L})^{-m} \in  
W_{2(\dim Z - m)} K_0(\mathbf{pHS}) \subseteq W_{2d} K_0(\mathbf{pHS}).
\]
So, $\chi_{\mathrm{Hdg}}^{\prime}(F_{\infty}\mathcal{M}_{\mathbb{C}}^{\prime}) \subseteq \bigcap_{d \in \mathbb{Z}} W_d K_0(\mathbf{pHS})$. But $\bigcap_{d \in \mathbb{Z}} W_d K_0(\mathbf{pHS}) = 0$ and hence there exists a ring homomorphism
\[
\overline{\mathcal{M}_{\mathbb{C}}^{\prime}} \rightarrow K_0(\mathbf{pHS})
\]
through which \eqref{chiHdg} factorizes.
If now $X$ and $Y$ are smooth projective such that $[X] = [Y]$ holds in $\widehat{\mathcal{M}_{\mathbb{C}}^{\prime}}$, then  the rational Hodge structures of $X$ and $Y$ have the same class in  $K_0(\mathbf{pHS})$. Finally, since the category $\mathbf{pHS}$ is semisimple, one has that these two Hodge structures are indeed isomorphic.

The equality of the Hodge numbers follows directly from Theorem \ref{invminimal} much more easily.
Indeed, 
the Hodge-Deligne polynomial $\mathrm{HD}^{\prime} \colon K_0^{\prime}(\mathrm{Var} / \mathbb{C}) \rightarrow \mathbb Z[u, v]$ appearing in \S \ref{final} extends to a ring homomorphism
\[
\widehat{\mathcal{M}_{\mathbb{C}}^{\prime}} \rightarrow \mathbb{Z}[\![u, v, (uv)^{-1}]\!].
\]
So, from Theorem \ref{invminimal}, it follows  that
\[
\mathrm{HD}^{\prime}([X]) =  \mathrm{HD}^{\prime}([Y]),
\]
 and therefore all Hodge numbers of $X$ and $Y$ are equal.  \end{proof}

\subsection{Gorenstein volume of a log terminal variety. } In this subsection and in the next one we recall some background involved in the proof of Theorems \ref{invminimal} and  \ref{thm:minimal}. Note that the former is a particular case of the latter.

Let 
$W$ be an algebraic variety with log terminal singularities, and let
$\widetilde W \rightarrow W$ be a log resolution of $W$ of relative canonical divisor
\[
K_{\widetilde{W}/ W} = \sum_{i \in I} a_i F_i\>.
\] 
Let $m$ be a positive integer such that $m K_W$ is  Cartier. The \emph{Gorenstein volume} of $W$ is defined by the formula
\[
\mu_{\mathrm{Gor}}(W) := \mathbb{L}^{- \dim W} \sum_{J \subseteq I} [F_J] \prod_{j \in J} \big( \frac{\mathbb{L} - 1}{\mathbb{L}^{a_j+1} -1} - 1 \big) \, \in \, \widehat{\mathcal{M}_{\mathbb{C}}}[\mathbb{L}^{\frac{1}{m}}]:= \widehat{\mathcal{M}_{\mathbb C}}[T]/(T^m -\mathbb{L}),
\]
where, for every subset $ J \subseteq I$, we set $F_J := \cap_{j \in J} F_j$, and $[F_J]$ is its class in $K_0(\mathrm{Var} / \mathbb{C})$. 
(See also \cite[\S 7.3.4]{motivicint} for an equivalent definition involving motivic Igusa zeta functions.) It may be checked that the definition does not depend on the chosen log resolution. Moreover, in order to compute $\mu_{\mathrm{Gor}}(W)$, one can  take any resolution of $W$ satisfying the only requirement that its exceptional locus is a divisor such that all irreducible components $F_i$ \emph{ with non-zero discrepancies} $a_i$ are simple normal crossing (see \cite[Theorem 3.4]{batyrev}, \cite[Definition 2.1]{batyrev2} and \cite{craw}).  
Note that  if $W$ admits a   \emph{crepant} resolution 
$T \rightarrow W$,  then
\begin{equation}\label{crepres0}
\mathbb{L}^{- \dim T} [ T ]  = \mu_{\mathrm{Gor}}(T) = \mu_{\mathrm{Gor}}(W),
\end{equation} 
where the first equality follows from the smoothness of $T$. Moreover, if $K_W$ is a Cartier divisor, i.e., $W$ is Gorenstein, then $\mu_{\mathrm{Gor}}(W) \in \widehat{\mathcal{M}_{\mathbb{C}}}$  by definition.

Finally, we denote by $\mu^{\prime}_{\mathrm{Gor}}(W)$ the image of the Gorenstein volume of $W$
via
 the ring homomorphism
\begin{equation}\label{gorprime}
\widehat{\mathcal{M}_{\mathbb{C}}}[\mathbb{L}^{\frac{1}{m}}] \rightarrow \widehat{\mathcal{M}_{\mathbb{C}}^{\prime}}[\mathbb{L}^{\frac{1}{m}}]
\end{equation}
induced by \eqref{ringhom1}.

\subsection{Relative abundance conjecture}\label{relabconj}
Next, we address the proof of Theorem \ref{invminimal}. We first recall that, under the assumption of Theorem \ref{invminimal}, the abundance conjecture is known to hold true. Indeed, since  $\omega_X$ is $a_X$-big and nef,   it is $a_X$-semiample by the relative basepoint-free theorem \cite[Theorem 3.24]{kollarmori}.\footnote{By \cite[Corollary 3.8]{Zhu}, $\omega_X$ is indeed semiample.}
 Therefore,  there exists a natural birational \emph{morphism}
\[
\vartheta \colon X \rightarrow  X^{can}_{a}.
\]
Since $X$ is smooth (hence Gorenstein), it follows that the canonical divisor of $X^{can}_a$  is Cartier.\footnote{We thank Fabio Bernasconi for having pointed out this fact (and its proof) to us.} Indeed the canonical divisor $K_X$ of $X$ is Cartier, $\vartheta$-nef and $\vartheta$-big (as $\vartheta$ is birational). Hence, by the relative basepoint-free theorem \cite[Theorem 3.24]{kollarmori}, $dK_X$ is $\vartheta$-free for \emph{all} $d \gg 0$. In particular, $K_X = (d+1)K_X - dK_X \sim \vartheta^* D$ for a Cartier divisor $D$ on the relative canonical model  of $X$. So $K_{X^{can}_a} = {\vartheta}_*K_X \sim D$ and $K_{X^{can}_a}$ is Cartier.

Moreover, $\vartheta$ is  a \emph{crepant} resolution.  This is a consequence of the 
 nefness of $\omega_X$
together with the
fact that $X^{can}_{a}$ has canonical singularities, \emph{i.e.}, the relative canonical divisor $K_{X/X^{can}_{a}} \geq 0$ (see \cite[Theorem 0-3-12]{kamama}). Indeed, if $K_{X/X^{can}_{a}} > 0$, then by the Hodge  index theorem (see, e.g., \cite[Lemma 3.6.2 (1)]{bchm}) there exists a curve $C \subseteq X$ contracted by $\vartheta$ such that $(K_X \cdot C) = (K_{X/X^{can}_{a}} \cdot C)  < 0$. This gives a contradiction.
Therefore, by \eqref{crepres0},
\begin{equation}\label{eqcartier}
\mu_{\mathrm{Gor}}(X^{can}_{a})  = \mathbb{L}^{- \dim X}  [X]
\end{equation}
holds in $\widehat{\mathcal{M}_{\mathbb{C}}}$.
Note that the same holds for $Y$. Indeed, as already noted, $\omega_Y$ is $a_Y$-big,
and it is also nef by \cite[Theorem 1.4]{ka2}.

\subsection{Proof of Theorem \ref{invminimal}}
 By \eqref{eqcartier} and \eqref{gorprime}, we have that
\begin{equation}\label{cregor}
\mu^{\prime}_{\mathrm{Gor}}(X^{can}_{a}) = \mathbb{L}^{- \dim X} [X]
\end{equation}
 in the ring $\widehat{\mathcal{M}_{\mathbb{C}}^{\prime}}$ and the same happens for $Y$. In order to conclude the proof of Theorem \ref{invminimal} it remains to verify that
\begin{equation}\label{gorequiv}
\mu^{\prime}_{\mathrm{Gor}}(X^{can}_{a}) = \mu^{\prime}_{\mathrm{Gor}}(Y^{can}_{a})
\end{equation}
holds in $\widehat{\mathcal{M}_{\mathbb{C}}^{\prime}}$. This will be the content of the next subsection.

\subsection{Proof of \eqref{gorequiv}}
We prove, more generally, the following theorem.
\begin{theorem}\label{gorequivthm}
Let $X$ and $Y$ be derived equivalent varieties. If the general fiber of the Albanese morphism of $X$ is of general type, then \eqref{gorequiv} holds true in $\widehat{\mathcal{M}_{\mathbb{C}}^{\prime}}[\mathbb{L}^{\frac{1}{m}}]$, where  $m \gg 0$ is such that $mK_{X^{can}_{a}}$ and $mK_{Y^{can}_{a}}$ are both Cartier divisors.\footnote{Note that, without the nefness assumption on $\omega_X$,  it may happen that $X^{can}_a$ is $\mathbb Q$-Gorenstein but not Gorenstein.}
\end{theorem}
\noindent This can be done quite similarly to the proof of Theorem \ref{thmGroclass}.
Keeping the notation of \S \ref{E1}, we begin with
some preparatory lemmas.
\begin{lemma}\label{lemiso}
If $\omega_X$ is $a_X$-big, then
\[
X^{can}_{a} \simeq \Proj \big( p_X^* \U_{X} \big) 
\]
over $\Alb X$.
\end{lemma}
\begin{proof}
By going back to the proof of Theorem \ref{thmGroclass}, we have already observed that the two algebras
$\mathcal R(a_X)$ and
$p_X^*\mathcal U_X$
only differ for the twisting of the torsion line bundles appearing in the finite semigroup $\{P_{\alpha_{m, j}}\}_{m \geq 1, j}$. Since the $\alpha_{m, j}$'s are torsion points of $\Alb X$, there exists an  integer $k_X$ big enough such that $\mu_{k_X}^* P_{\alpha_{m, j}} \cong  \OO_{\Alb X}$ for all $m$ and  $j$, where $\mu_{k_X} \colon \Alb X \rightarrow \Alb X$ is the multiplication-by-$k_X$ isogeny of $\Alb X$. Therefore, 
$\mu_k^*\mathcal R(a_X)$ and
$\mu_k^*p_X^*\mathcal U_X$  are isomorphic in degrees $\geq 1$. Then,
 from \cite[Prop.\ 3.1.8 $(ii)$]{egaII} and the functoriality of the $\Proj$ construction \cite[Prop.\ 3.5.3]{egaII}, it follows
 that $X^{can}_{a} \rightarrow \Alb X$ and $\Proj \big( p_X^* \U_{X} \big) \rightarrow \Alb X$ become isomorphic to the same variety $Z$ after base change with the isogeny $\mu_{k_X}$. Namely, we have cartesian commutative diagrams
\[
\xymatrix{Z \ar[d]_-{\mu_{k_X}'} \ar[r] &\Alb X \ar[d]^-{\mu_{k_X}}  &Z \ar[l] \ar[d]^-{\mu_{k_X}''} \\
X^{can}_{a} \ar[r] &\Alb X  &\Proj \big( p_X^* \U_{X} \big) \ar[l]
}
\]
In particular,  since $X^{can}_{a}$ has canonical singularities,  $\Proj \big( p_X^* \U_{X} \big)$ has canonical singularities as well (see \cite[Proposition 2.15]{kollarbook}).
Moreover, since  $K_{X^{can}_{a}}$ is relatively ample with respect to   $X^{can}_{a} \rightarrow \Alb X$, 
one has that  
 the canonical divisor of $\Proj \big( p_X^* \U_{X} \big)$ is relatively ample with respect to  its structure morphism $\Proj \big( p_X^* \U_{X} \big) \rightarrow \Alb X$ because $\mu_{k_X}$ (and hence $\mu_{k_X}'$ and $\mu_{k_X}''$) is an \'etale finite surjective  morphism.
By the uniqueness of the relative canonical model \cite[Theorem 0-3-12]{kamama},  $X^{can}_{a}$ and $\Proj \big( p_X^* \U_{X} \big)$ are therefore isomorphic over $\Alb X$.
\end{proof}

We get the cartesian diagram
\begin{equation}\label{0logres}
\xymatrix{
X^{can}_{a}    \ar[r]^-{\tau_{X}} \ar[d]^{p_X'} & \Alb X \ar[d]^{p_X}  \\
\Proj \big(  \U_{X} \big)    \ar[r]^-{  \sigma_X } & \Alb Z_X }
\end{equation}

\begin{lemma}\label{lemmalogres}
If $T$ is a log resolution of $\Proj \big( \U_{X} \big)$, then the corresponding base change via the morphism $p_X'$ defined in \eqref{0logres} gives a  resolution $\widetilde T$ of $X^{can}_{a}$ such that the exceptional divisors with non-zero discrepancies are simple normal crossing.\footnote{Actually, it can be proved that $\widetilde T$ is a log resolution of $X^{can}_{a}$. However this is unnecessary for our purposes.}
\end{lemma}
\begin{proof}
 We have the following cartesian diagrams by construction
\[
\xymatrix{
\widetilde T    \ar[r]^-{\widetilde{\rho}} \ar[d]^{\nu_X} & X^{can}_{a}  \ar[d]^{p_X'}  \ar[r]^{\tau_X} &\Alb X \ar[d]^{p_X} \\
T   \ar[r]^-{\rho} & \Proj \big(  \U_{X} \big) \ar[r]^-{\sigma_X} &\Alb Z_X  }
\]
and $\widetilde T$
 is a resolution of $X^{can}_{a}$. We only need to prove that the exceptional divisors of $\widetilde{\rho}$ with non-zero discrepancies are simple normal crossing. 
Note that 
\begin{equation}\label{pbexdiv}
\nu_X^* K_{\rho} = K_{\widetilde{\rho}},
\end{equation}
where $K_{\rho}$ and $K_{\widetilde{\rho}}$ are the relative canonical divisors of $\rho$ and $\widetilde{\rho}$, respectively (see \cite[Proof of Proposition 2.7]{chou} for an even more general formula based on \cite{defernexhacon}).
Therefore, every  $\widetilde{\rho}$-exceptional divisor that is not a component of $\nu_X^*F$ for any $\rho$-exceptional divisor $F$ has discrepancy equal to $0$.  
The Lemma follows because by assumption $\rho$ is a log resolution, and the pullback of a normal crossing divisor is a normal crossing divisor
 (\cite[Tag 0CBP]{stacksproject}).
\end{proof}

Now we are ready to prove \eqref{gorequiv}. 
Take $\nu_X \colon \widetilde T \rightarrow T$ as in Lemma \ref{lemmalogres} and write
\[
K_{\rho} = \sum_{i \in I} a_i F_i \quad \textrm{and} \quad   K_{\widetilde{\rho}} = \sum_{i \in I} a_i F_{X, i},
\]
where $F_{X, i} := \nu_X^* F_i = \nu_X^{-1}(F_i)$ (see \eqref{pbexdiv}). Note that, since $\nu_X$ is a smooth morphism, the $F_{X, i}$'s are reduced for all $i$. 
The same holds for $Y$, with similar notations. Indeed,  recall that 
$\Proj\bigl(\mathcal U_X\bigr)\rightarrow \Alb Z_X$ is a derived invariant thanks to Theorem \ref{thmbasechange}(2).
Therefore, 
\[
\mu_{\mathrm{Gor}}(X^{can}_{a}) = \mathbb{L}^{- n} \sum_{J \subseteq I} [F_{X, J}] \prod_{j \in J} \big( \frac{\mathbb{L} - 1}{\mathbb{L}^{a_j+1} -1} - 1 \big),
\]
and
\[
\mu_{\mathrm{Gor}}(Y^{can}_{a}) = \mathbb{L}^{- n} \sum_{J \subseteq I} [F_{Y, J}] \prod_{j \in J} \big( \frac{\mathbb{L} - 1}{\mathbb{L}^{a_j+1} -1} - 1 \big),
\]
where $n = \dim X^{can}_{a} = \dim Y^{can}_{a}$.
As in  \eqref{gro2}, we get
$[F_{X, J}] = [\ker (p_X)] [F_J]$
and $[F_{Y, J}] = [\ker (p_Y)] [F_J]$.
Hence
the equality
$[F_{X, J}] = [F_{Y, J}]$ 
 in $K_0^{\prime}(\mathrm{Var}/\mathbb{C})$ follows from Claim \ref{claim} as in the proof of Theorem \ref{thmGroclass}. In particular,  $\mu_{\mathrm{Gor}}^{\prime}(X^{can}_{a}) = \mu_{\mathrm{Gor}}^{\prime}(Y^{can}_{a})$. 
\qed

\subsection{Relative minimal models. }\label{minrelmodel}
 Given a variety $X$ such that $\omega_X$ is $a_X$-big, let  $X^{min}_{a}$ be a  relative minimal model of $X$ with respect to $a_X$,
whose existence  is guaranteed by \cite[Theorem 1.2]{bchm}. 
 Contrary to the relative canonical model, $X^{min}_{a}$ 
 is not necessarily unique. However, two $a_X$-relative minimal models of $X$ are $K$-equivalent by \cite[Variant 1.11 of Theorem 1.4]{wang} and, therefore,  they have the same Gorenstein volume \cite[Proposition 1.2]{yasuda}. We use a similar notation for $Y$. 
The above argument  proves the following  result, valid under hypotheses more general than Theorem \ref{invminimal}:
\begin{theorem}\label{thm:minimal}
Let $X$ and $Y$ be derived equivalent varieties. If $\omega_X$ is $a_X$-big, then
the equality 
\begin{equation}\label{gengor}
\mu_{\mathrm{Gor}}^{\prime}(X^{min}_{a}) = \mu_{\mathrm{Gor}}^{\prime}(Y^{min}_{a})
\end{equation}
holds  in
$\widehat{\mathcal{M}_{\mathbb{C}}^{\prime}}[\mathbb{L}^{\frac{1}{m}}],$
for $m \gg 0$.
\end{theorem}
Indeed, as in \S \ref{relabconj}, we have a birational morphism $f \colon X^{min}_{a} \rightarrow X^{can}_{a}$ such that $K_{X^{min}_{a}} = f^* K_{X^{can}_{a}}$. Hence
$\mu_{\mathrm{Gor}}(X^{min}_{a}) = \mu_{\mathrm{Gor}}(X^{can}_{a})$, and the same happens for $Y$. 
Therefore, \eqref{gengor} follows from Theorem \ref{gorequivthm}.

\section{Further results, variants, questions}\label{ultima}

\subsection{Rouquier-stable morphisms }\label{ultima1}

Going back to the terminology of Subsection \ref{R1}, given an equivalence  $\Phi_{\E} \colon D(X) \to D(Y)$, an abelian subvariety $\widehat{B_X}$ of $\Pic0 X$ is said to be  \emph{R-stable}  if its points are $R$-stable with respect to  the Rouquier isomorphism $\varphi_\E$. Therefore the Rouquier isomorphism induces an isomorphism of abelian varieties	
\[
\varphi_\E\colon\widehat{B_X}\buildrel\sim\over\longrightarrow \widehat{B_Y}
\]
where $\widehat{B_Y}$ is the abelian subvariety of $\Pic0 Y$ such that $\varphi(\mathrm{id}_X,\widehat{B_X})=(\mathrm{id}_Y,\widehat{B_Y})$. 
Note that, in this terminology, the Albanese-Iitaka variety is the dual of an $R$-stable subvariety of $\Pic0 X$ (Subsection \ref{R3}). It turns out that Theorem \ref{step1} holds more generally, with the same proof, for any morphism to an abelian variety which is the dual of an
 $R$-stable abelian subvariety of $\Pic0 X$.To be precise, 
as in \eqref{xy},
an $R$-stable subvariety $\widehat{B_X}$ gives commutative diagrams
 \begin{equation}\label{xyvar}
 \xymatrix{ X \ar[r]^-{a_X} \ar[rd]^-{b_X} \ar[d]^-{s_X} & \Alb X \ar[d]^-{p_X} \\ 
 X' \ar[r]^-{b'_X}   & B_X} \qquad \qquad
 \qquad \xymatrix{ Y \ar[r]^-{a_Y} \ar[rd]^-{b_Y} \ar[d]^-{s_Y} & \Alb Y \ar[d]^-{p_Y} \\ 
 Y' \ar[r]^-{b'_Y} & B_Y} 
 \end{equation}
   where $p_X$ (\emph{resp}.\ $p_Y$) is the dual morphism of the inclusion  $\widehat{B_X} \subseteq \Pic0 X$ (\emph{resp}.\ $\widehat{B_Y} \subseteq \Pic0 Y$) and 
	the left-bottom sides of the diagrams
are the Stein factorizations of the composed morphisms $b_X:=p_X\circ a_X$ and $b_Y:= p_Y\circ a_Y$.  

\begin{theorem}\label{thm:step1-gen}
In the above notation, given   an $R$-stable subvariety $\widehat{B_X} \subset \Pic0 X$ (with respect to the equivalence $\Phi_\E$), there is an induced
 isomorphism $\psi \colon Y^\prime\rightarrow X^\prime$  such that following diagram 
\begin{equation}\label{xy2}
\xymatrix{X' \ar[d]^{b_X'} &Y' \ar[l]^{\psi}_\sim\ar[d]^{b_Y'}\\
B_X & B_Y \ar[l]^{\widehat{\varphi_\E}}_\sim}
\end{equation}
is commutative. 
\end{theorem} 

Also, all the results of \S \ref{H} hold more generally, with the same proof, replacing the Albanese-Iitaka morphism with any morphism 
$b_X \colon X \to B_X$ to the dual of an $R$-stable abelian subvariety. In particular, we have the following generalization of Theorem \ref{thm:invmult}.
\begin{theorem}\label{thm:invmult-new} There is an isomorphism  of relative canonical algebras
\[
\widehat{\varphi_\E}^*\mathcal R(b_X)\cong \mathcal R(b_Y).
\]
\end{theorem}

An interesting problem is to determine the maximal derived invariant subvariety of $\Pic0 X$, namely the maximal $R$-stable 
abelian subvariety with respect to any derived equivalence of $X$ (by \S\ref{R3} this contains the Albanese-Iitaka variety $\Alb Z_X$).  The following lemma provides a sufficient condition to ensure that such maximal  variety is the full $\Pic0 X$.

\begin{lemma}\label{lemmaaffine}
If $\Aut0 X$ is an affine algebraic group, then $\Pic0 X$ is  $R$-stable.
\end{lemma}
\begin{proof}
If $\Aut0 X$ is affine, then $\Aut0Y$ is affine too (see \cite[p.\ 534]{ps1}), and,  since $\Pic0 X$ is projective, the composed morphism
\[
\{\mathrm{id}_X \} \times \Pic0 X \xrightarrow{\varphi_{\E}} \Aut0 Y \times \Pic0 Y \xrightarrow{p_1} \Aut0 Y
\]
is   constant, where $p_1$ is the  projection onto the first factor. 
\end{proof}
For instance, $\Aut0 X$ is affine  when
$\chi(X, \OO_X) \neq 0$ by \cite[Corollary 2.6]{ps1}. In conclusion, in this case the Albanese map itself satisfies the hypotheses of Theorem \ref{thm:invmult-new}, hence Theorem \ref{C} in the Introduction is greatly simplified, as we have that the Albanese relative canonical ring is preserved by the Rouquier isomorphism

\begin{corollary}\label{cor:invariance of ra} If $\chi(X,\OO_X)\ne 0$, then 
\[
\widehat{\varphi_\E}^*\mathcal R(a_X)\cong \mathcal R(a_Y).
\]
\end{corollary}

Connected to this, another related interesting line of research, introduced in the paper \cite{liol}  in a wider context, consists in a systematic study of those derived equivalences $\Phi_\E\colon D(X)\rightarrow D(Y)$ such that the full $\Pic0 X$ is $R$-stable. In the terminology of \cite{liol}, these are \emph{strongly filtered equivalences}. An in-depth study of such derived equivalences and of invariant properties for morphisms to abelian varieties dual to $R$-stable abelian subvarieties of $\Pic0 $ will be carried out in the forthcoming paper \cite{calo}.

\subsection{Derived invariance of refined non-vanishing canonical loci. }
For each $k \geq 1$, let
\[
V^i_k(\omega_X) := \{ \alpha \in \Pic0 X \ |\ h^i(X, \omega_X \otimes P_{\alpha}) \geq k \}
\]
 and denote by $V^i_k(\omega_X)_0$ the union of the irreducible components of $V^i_k(\omega_X)$ passing through the origin (such irreducible components are in fact abelian subvarieties).  Conjectures of Popa and the second author (\cite{popa-conj1}, \cite{popa-conj2}) predict, in the present language,  that such abelian subvarieties are $R$-stable with respect to any derived equivalence and, moreover $h^i(X, \omega_X \otimes P_{\alpha})=h^i(Y, \omega_Y \otimes Q_{\varphi_\E(\alpha)}) $ for all $\alpha\in V^i_1(\omega_X)_0$. 
 Via methods introduced in \cite{popa-conj2},  when the canonical bundle is $a_X$-big this follows from Theorem \ref{invHN1}.
\begin{theorem}\label{popa-conj}
Suppose $\Phi_{\E} \colon D(X) \to D(Y)$ is an equivalence and that $\omega_X$ is $a_X$-big.
Then, for all integers $i\geq 0$ and $k\geq 1$, the varieties  $V^i_k(\omega_X)_0$ are $R$-stable with respect to $\Phi_\E$. In other words, the  Rouquier isomorphism $\varphi_{\E}$ acts on $V^i_k(\omega_X)_0$ as 
$$\varphi_{\E} \big( {\rm id}_X , V^i_k(\omega_X)_0 \big) \; = \; \big( {\rm id}_Y , V^i_k(\omega_Y)_0 \big).$$
\end{theorem}

\begin{proof}
The proof  follows from the proof of \cite[Theorem 3.1]{popa-conj2}, where it is shown that the assertion follows from the invariance of the Hodge numbers $h^{0,i}$ of  all derived equivalent varieties $X_\alpha$ and $Y_\beta$, which are \'etale covers of $X$ (\emph{resp}. $Y$), induced by torsion points $\alpha\in\Pic0 X$ (\emph{resp}. $\beta\in \Pic0 Y$). But the $a_X$-bigness of the canonical bundle of $X$ implies the same property
 for $X_\alpha$. Hence, by Theorem \ref{invHN1}, $h^{0,i}(X_\alpha)=h^{0,i}(Y_\beta)$ as soon as there is an equivalence  $D(X_\alpha)\rightarrow D(Y_\beta)$.
\end{proof}

\section{Appendix: the Chen-Jiang decomposition and the torsion filtration}\label{rem:support} 
 Chen-Jiang decompositions as \eqref{eq:CJ} hold more generally 
 for pushforwards of pluricanonical sheaves under any morphism $f\colon X\rightarrow A$ 
 from a smooth projective complex variety to an abelian variety (\cite[Theorem C]{loposc}). 
 They appear to be quite useful in the study of the birational geometry of irregular varieties. 
 Here we point out a useful feature of them, used in the proof of Proposition \ref{prop:pop}.

More generally, let $\F$ be a coherent 
sheaf on an abelian variety $A$ having 
 \emph{the Chen-Jiang decomposition property} (this notion is introduced in \cite[\S 3]{loposc}), \emph{i.e.} admitting a direct-sum decomposition as in \eqref{eq:CJ}, namely
\begin{equation}\label{FCJ}
\F \; \cong \;    \bigoplus_i \rho_{i}^* \mathcal{F}_{i} \otimes P_{\gamma_{i}}, 
\end{equation}
satisfying the four properties listed in the proof of Proposition \ref{prop:simplified}, namely:
\begin{itemize}
\item[(i)]  each $\rho_{i} \colon A \rightarrow B_{i}$ is a quotient morphism of abelian varieties with connected fibers;\\
\item[(ii)] each $\mathcal{F}_{i}$ is a nonzero $M$-regular sheaf  on $B_{i}$ (see \S\ref{R4});\\
\item[(iii)] each $\gamma_{i} \in \Pic0 X$ is a torsion point;\\
\item[(iv)]  $\rho_{i}^*\Pic0 B_{i}-\gamma_{i} \ne   \rho_{j}^*\Pic0 B_{j} -\gamma_{j}$\  for $i\ne j$.
\end{itemize}

 Applying $\FM_{A}$ to \eqref{FCJ} we get that
\begin{equation}\label{FM(FCY)}
\reallywidehat{\F^\vee} \; \cong \;  \bigoplus_{i}t_{\gamma_{i}*}({\widehat{\rho}}_{i*}\reallywidehat{\F_{i}^\vee}).
\end{equation}
 This follows from  the commutativity relation between FMP transforms and morphisms of abelian varieties  $\pi\colon A\rightarrow B$ 
\begin{equation}\label{eq:commutation1}
 \FM_A\circ {\bf L}\pi^* \; \cong \; {\bf R}\hat{\pi}_*\circ \FM_B
\end{equation}
(\cite[Proposition 4.1]{schnell}), and the formula
\begin{equation}\label{eq:translation}
\FM_A( - \otimes P_\beta) \; \cong \; t_{\beta*}\circ \FM_A \quad \forall \beta \in \Pic0 A
\end{equation}
(\cite[Proposition 5.1]{schnell}). 
Note that, since the sheaves $\F_i$ are $M$-regular, the sheaves $t_{\gamma_{i}*}({\widehat{\rho}}_{i*}\reallywidehat{\F_{i}^\vee})$ are of pure dimension equal to $\dim B_i$ and their scheme-theoretic support are the translated subtori 
\begin{equation}\label{eq:scheme-supp}
-V^0 \big(t_{\gamma_i}^*(\rho_i^*\F_i) \big) \; = \; \rho_i^*\Pic0 B_i+ \gamma_i
\end{equation}
with the reduced scheme structure.

We consider the torsion filtration of the sheaf $\reallywidehat{\F^\vee}$
\begin{equation}\label{torsion-fil}
0 \; \subset \; T_0(\reallywidehat{\F^\vee}) \; \subset \; \cdots \; \subset\; T_d(\reallywidehat{\F^\vee}) \; = \; \reallywidehat{\F^\vee}
\end{equation}
where $d=\dim \reallywidehat{\F^\vee}$ and $T_i(\reallywidehat{\F^\vee})$ is the maximal subsheaf of $\reallywidehat{\F_i^\vee}$ of dimension $\le i$  (see \cite[\S1.1]{huybrechts-lehn}). It follows that $d$ is equal to the maximal dimension of the tori $B_i$. In particular $d=\dim A$ if and only if there is a (necessarily unique, by condition (iv)) $i$ such that $B_i=A$. 

\begin{proposition}
For all $k\le d$ we have that
\begin{equation}\label{eq:T1}
T_k \big( \reallywidehat{\F^\vee} \big) \; \cong \; \bigoplus_{\dim B_i\le k} t_{\gamma_{i}*} \big( 
{\widehat{\rho}}_{i*}\reallywidehat{\F_{i}^\vee} \big) 
\end{equation}
and
\begin{equation}\label{eq:T2}
T_k \big( \reallywidehat{\F^\vee} \big) \; \cong \;  \bigl(T_{k-1} \big( \reallywidehat{\F^\vee} \big) \bigr) \oplus
 \bigl(T_{k} \big( \reallywidehat{\F^\vee} \big) / T_{k-1} \big( \reallywidehat{\F^\vee} \big) \bigr) 
 \cong \bigl(\bigoplus_{\dim B_i\le k-1} t_{\gamma_{i}*} \big( {\widehat{\rho}}_{i*}\reallywidehat{\F_{i}^\vee} \big) \bigr)
 \oplus\bigl(\bigoplus_{\dim B_i= k} t_{\gamma_{i}*} \big( {\widehat{\rho}}_{i*}\reallywidehat{\F_{i}^\vee} \big) \bigr).
\end{equation}
\end{proposition}
\begin{proof}
For $k=d$ \eqref{eq:T1} is obvious while  \eqref{eq:T2} holds because there is an inclusion of sheaves
\[
\bigoplus_{\dim B_i\le d-1} t_{\gamma_{i}*} \big( {\widehat{\rho}}_{i*}\reallywidehat{\F_{i}^\vee} \big) \subset T_{d-1} \big( \reallywidehat{\F^\vee} \big)
\]
 and the natural  map 
 \[
 \bigoplus_{\dim B_i= d} t_{\gamma_{i}*} \big( {\widehat{\rho}}_{i*}\reallywidehat{\F_{i}^\vee} \big) \rightarrow T_{d} \big( \reallywidehat{\F^\vee} \big)/ T_{d-1} \big( \reallywidehat{\F^\vee} \big) 
 \] 
 is injective. For  $k<d$, \eqref{eq:T1} holds by induction and \eqref{eq:T2} follows as above. 
\end{proof}
In particular, it follows that, via the symmetric FMP transform $FM_A$, the Chen-Jiang decomposition  of a sheaf $\F$ is recovered by  the torsion filtration of $\reallywidehat{\F^\vee}$, and conversely.

\providecommand{\bysame}{\leavevmode\hbox
to3em{\hrulefill}\thinspace}


\begin{thebibliography}{EMS}






\bibitem[Ba1]{batyrev} V. V. Batyrev, {Stringy Hodge numbers of varieties with Gorenstein canonical singularities}, in {\em Integrable systems and algebraic geometry (Kobe/Kyoto, 1997)}, 1--32, World Sci. Publ., River Edge, NJ, 1998.


\bibitem[Ba2]{batyrev2} V. V. Batyrev, 
{Stringy Hodge numbers and Virasoro algebra},
Math. Res. Lett. \textbf{7} (2000), no. 2-3, 155--164.

\bibitem[BCHM]{bchm} C. Birkar, P. Cascini, Ch. Hacon and J. McKernan, {Existence of minimal models for varieties of log general type}, 
J. Amer. Math. Soc. \textbf{23} (2010), no. 2, 405--468.

\bibitem[BC]{birkarchen} C. Birkar and J. A. Chen,
{Varieties fibred over abelian varieties with fibres of log general type},
Adv. Math. \textbf{270} (2015), 206--222.

\bibitem[Cal]{calda} A. C\v{a}ld\v{a}raru, {The Mukai pairing, II: the Hochschild-Kostant-Rosenberg isomorphism}
Adv. in Math. \textbf{194} (2005)  34--66 




\bibitem[CL]{calo} F. Caucci and L. Lombardi, {Irregular fibrations of derived equivalent varieties}, arXiv:2207.14081v2 (2022)


\bibitem[CP]{capa} F. Caucci and G. Pareschi, {Derived invariants arising from the Albanese map}, Algebr. Geom. \textbf{6} (2019), no. 6, 730--746.


\bibitem[C-LNS]{motivicint}
 A. Chambert-Loir, J. Nicaise and J. Sebag, {\em Motivic integration},
Progress in Mathematics, \textbf{325}, Birkh\"auser/Springer, New York, 2018.


\bibitem[CH1]{ch-iitaka1} J. A. Chen and Ch. Hacon, {Pluricanonical maps of varieties of maximal Albanese dimension}, Math. Ann. \textbf{320}  (2001), 367--380.

\bibitem[CH2]{chenhaconirr} J. A. Chen and Ch. Hacon, {On the irregularity of the image of the Iitaka fibration},
Comm. Algebra \textbf{32} (2004), no. 1, 203--215.

\bibitem[CJ]{cj} J. A. Chen and Z. Jiang, {Positivity on varieties of Albanese dimension},
J. Reine Angew. Math. \textbf{736} (2018), 225--253.




\bibitem[Ch]{chou} C.-C. Chou, {A transversality theorem for some classical varieties},
J. Pure Appl. Algebra \textbf{217} (2013), no. 12, 2274--2281.



\bibitem[Cr]{craw} A. Craw, 
{An introduction to motivic integration} in {\em Strings and geometry}, 203--225,
Clay Math. Proc., \textbf{3}, Amer. Math. Soc., Providence, RI, 2004.




\bibitem[dFH]{defernexhacon} T. de Fernex and Ch. D. Hacon, 
{Singularities on normal varieties},
Compos. Math. \textbf{145} (2009), no. 2, 393--414.






\bibitem[EGA II]{egaII} A. Grothendieck, {\'{E}l\'{e}ments de G\'{e}om\'{e}trie Alg\'{e}brique, II, \'{E}tude globale \'{e}l\'{e}mentaire de quelques classes de morphismes}, Publ. Math. IHES \textbf{8} (1961).


\bibitem[EGA IV$_2$]{egaiv} A. Grothendieck, {\'{E}l\'{e}ments de g\'{e}om\'{e}trie alg\'{e}brique, IV, \'{E}tude locale des sch\'{e}mas et des morphismes de sch\'{e}mas, II}, Inst. Hautes \'{E}tudes Sci. Publ. Math. No. \textbf{24} (1965).




\bibitem[F]{fujita} T. Fujita, {Cancellation problem of complete varieties}, 
Invent. Math. \textbf{64} (1981), no. 1, 119--121.



\bibitem[GL]{gl2}
M. Green and R. Lazarsfeld, {Higher obstructions to deforming cohomology groups of line bundles}, J. Amer. Math. Soc. \textbf{1} (1991),  87--103.


\bibitem[Ha]{ha} Ch. Hacon, {A derived category approach to generic vanishing}, J. Reine Angew. Math. \textbf{575} (2004), 173--187.

\bibitem[HaPa]{hp} Ch. Hacon and R. Pardini, 
{On the birational geometry of varieties of maximal Albanese dimension},
J. Reine Angew. Math. \textbf{546} (2002), 177--199.

\bibitem[HaPoS]{hps} Ch. Hacon, M. Popa, Ch Schnell, {Algebraic fibers spaces over abelian varieties: around a recent theorem by Cao and P\u{a}un}, in
{\em Local and global methods in algebraic geometry}, 
Contemp. Math. \textbf{712}, Amer. Math. Soc., Providence, RI, (2018), 143--195.



\bibitem[Hu]{Zhu} Z. Hu, {Log canonical pairs over varieties with maximal
Albanese dimension}, Pure and Applied Mathematics Quarterly \textbf{12} (2016), no. 4, 543--571.


\bibitem[Huy]{huybrechts} D. Huybrechts, {\em Fourier-Mukai transforms in algebraic geometry}, Clarendon Press - Oxford (2006)

\bibitem[HuLe]{huybrechts-lehn} D. Huybrechts, Ch. Lehn, {\em The geometry of moduli spaces of sheaves}, second edition, Cambridge University Press (2010)

\bibitem[IMOU]{itoetal} A. Ito, M. Miura, S. Okawa and K. Ueda, {Derived equivalence and Grothendieck ring of varieties: the case of $K3$ surfaces of degree 12 and abelian varieties}, Selecta Math. (N.S.) \textbf{26} (2020), no. 3, Paper No. 38, 27 pp.







\bibitem[Ka]{ka2} Y. Kawamata, {$D$-equivalence and $K$-equivalence}, J. Diff. Geom. \textbf{61} (2002), 147--171.

\bibitem[KMM]{kamama} Y. Kawamata, K. Matsuda and K. Matsuki,  {Introduction to the minimal model
problem}, in  {\em Algebraic Geometry, Sendai, 1985}, T. Oda (ed.). Adv. Stud. Pure
Math. (1987), Vol. 10, pp. 283--360.








\bibitem[Ko]{kollarbook} J. Koll\'ar, 
{\em Singularities of the minimal model program. With a collaboration of S\'andor Kov\'acs}, Cambridge Tracts in Mathematics, \textbf{200}, Cambridge University Press, Cambridge, 2013.


\bibitem[KM]{kollarmori} J. Koll\'ar and S. Mori, {\em Birational geometry of algebraic varieties}, 
 Cambridge Tracts in Mathematics, \textbf{134}, Cambridge University Press, Cambridge, 1998.



\bibitem[Lai]{lai} C. J.  Lai, {Varieties fibered by good minimal models}, Math. Ann. \textbf{350} (2011), no. 3, 533--547.



\bibitem[LO]{liol} M. Lieblich and  M. Olsson, {Derived categories and birationality}, preprint arXiv:2001.05995v2 (2020).

 
\bibitem[Lo1]{lombardi-derived} L. Lombardi, {Derived invariants of irregular varieties and Hochschild homology,} Algebra Number
Theory \textbf{8} (2014), no. 3, 513--542.

\bibitem[Lo2]{lombardi-fibrations}L. Lombardi, {Derived equivalence and fibrations over curves and surfaces}, Kyoto J. Math. \textbf{62} (2022), no. 4, 683--706.


\bibitem[LoPo]{popa-conj2} L. Lombardi and M. Popa, {Derived equivalence and non-vanishing loci II}, 
{\em London Math. Soc. Lecture Note Series, 417, Recent Advances in Algebraic Geometry}, Cambridge University Press (2015) 291--306

\bibitem[LoPoS]{loposc} L. Lombardi, M. Popa and Ch. Schnell, Pushforwards of pluricanonical bundles under morphisms to abelian varieties, J. Eur. Math. Soc. (JEMS) \textbf{22} (2020), no. 8, 2511--2536.

\bibitem[MPo]{meng-popa} F. Meng and M. Popa, Kodaira dimension of fibrations over abelian varieties, preprint  arXiv:2111.14165v2[math.AG]


\bibitem[Mu]{mukai} S. Mukai, {Duality between D($X$) and D($\hat X$) with its application to Picard sheaves}, Nagoya Math. J. \textbf{81} (1981), 153--175.




\bibitem[O]{orlov} D. Orlov, {Derived categories of coherent sheaves and equivalences between them}, Russian Mathematical Surveys 
\textbf{58} (2003) 511--591

\bibitem[Pa1]{msri} G. Pareschi, {Basic results on irregular varieties via Fourier-Mukai methods}, in {\em Current Developments in Algebraic Geometry}, Math. Sci. Res. Inst. Publ., vol. 59 (Cambridge Univ. Press, Cambridge, 2012), 379--403.

\bibitem[Pa2]{Pa2} G. Pareschi, {Standard canonical support loci}, Rend. Circ. Mat. Palermo, II. Ser (2017), 137--157




\bibitem[PaPo2]{PP3} G. Pareschi and M. Popa, {Regularity on abelian varieties III: relationship with Generic  Vanishing and  applications},
in {\em Grassmannians, moduli spaces and vector bundles}, 141--167,
Clay Math. Proc., \textbf{14}, Amer. Math. Soc., Providence, RI, 2011.

\bibitem[PaPoS]{paposc} G. Pareschi, M. Popa and Ch. Schnell, {Hodge modules on complex tori and generic vanishing for
compact K\"ahler manifolds}, Geom. Topol. \textbf{21} (2017), no. 4, 2419--2460.


\bibitem[Po]{popa-conj1} M. Popa, {Derived equivalence and non-vanishing loci}, in {\em A Celebration of Algebraic Geometry}, Clay Math. Proc., vol. 18 (Amer. Math. Soc., Providence, RI, 2013), 567--575.

\bibitem[PoS1]{ps1} M. Popa and Ch. Schnell, {Derived invariance of the number of holomorphic 1-forms and vector fields}, 
Ann. Sci. \'Ecole Norm. Sup. \textbf{44} (2011), no. 3, 527--536.









\bibitem[PoS2]{ps3} M. Popa and Ch. Schnell, {On direct images of pluricanonical bundles}, Algebra \& Number Theory 
\textbf{8} (2014), no. 9, 2273--2295



\bibitem[R]{rouquier} R. Rouquier, {Automorphismes, graduations et cat\'egories triangul\'ees}, J. Inst. Math. Jussieu \textbf{10}
(2011)  713--751.

\bibitem[Sch]{schnell} Ch. Schnell, {The Fourier-Mukai transform made easy}, 	Pure Appl. Math. Q. \textbf{18} (2022), no. 4, 1749--1770.


\bibitem[S-P]{stacksproject}
The Stacks-project authors, {The stacks project}, https://stacks.math.columbia.edu, 2022.



\bibitem[V]{villadsen} M. Villadsen, {Chen-Jiang decompositions for projective varieties, without Hodge modules}, Math. Zeitschrift 
\textbf{300}  (2022), no. 3,  2099--2116.

\bibitem[Wa]{wang} C.-L. Wang, 
{On the topology of birational minimal models},
J. Differential Geom. \textbf{50} (1998), no. 1, 129--146.

\bibitem[Ya]{yasuda} T. Yasuda, 
{Twisted jets, motivic measures and orbifold cohomology},
Compos. Math. \textbf{140} (2004), no. 2, 396--422.

\end{thebibliography}
\end{document}